\documentclass[hidelinks,12pt,a4paper]{article}
\usepackage[utf8x]{inputenc}
\usepackage[top=30mm, bottom=30mm, inner=20mm, outer=20mm, headsep=10mm, footskip=12mm]{geometry}
\usepackage{amssymb,csquotes}
\usepackage{amsmath,bm}
\usepackage{mathtools}
\usepackage{amsthm}
\usepackage[english]{babel}
\usepackage{lmodern}
\usepackage[mathscr]{euscript}
\usepackage[colorlinks=true]{hyperref}
\usepackage{todonotes}
\usepackage{appendix}
\usepackage{cleveref}
\usepackage{bm}\usepackage{cancel}
\usepackage{upgreek,orcidlink,dirtytalk}
\usepackage[normalem]{ulem}
\usepackage{float}\usepackage{tikz,tikz-cd}

\usepackage{todonotes}

\usepackage[
    style=numeric,
    giveninits=true,
    isbn=false,
    doi=false,
    eprint=true,
    url=false,
    backend=biber,
    maxbibnames=99
]{biblatex}
\usetikzlibrary{patterns}
\numberwithin{equation}{section}

\theoremstyle{plain}
\newtheorem{definition}{Definition}[section]

\newtheorem{theorem}[definition]{Theorem}
\newtheorem{proposition}[definition]{Proposition}
\newtheorem{lemma}[definition]{Lemma}
\newtheorem{corollary}[definition]{Corollary}

\theoremstyle{definition}
\newtheorem{remark}[definition]{Remark}

\makeatletter
\let\save@mathaccent\mathaccent
\newcommand*\if@single[3]{%
  \setbox0\hbox{${\mathaccent"0362{#1}}^H$}%
  \setbox2\hbox{${\mathaccent"0362{\kern0pt#1}}^H$}%
  \ifdim\ht0=\ht2 #3\else #2\fi
  }
\newcommand*\rel@kern[1]{\kern#1\dimexpr\macc@kerna}
\newcommand*\widebar[1]{\@ifnextchar^{{\wide@bar{#1}{0}}}{\wide@bar{#1}{1}}}
\newcommand*\wide@bar[2]{\if@single{#1}{\wide@bar@{#1}{#2}{1}}{\wide@bar@{#1}{#2}{2}}}
\newcommand*\wide@bar@[3]{%
  \begingroup
  \def\mathaccent##1##2{%
    \let\mathaccent\save@mathaccent
    \if#32 \let\macc@nucleus\first@char \fi
    \setbox\z@\hbox{$\macc@style{\macc@nucleus}_{}$}%
    \setbox\tw@\hbox{$\macc@style{\macc@nucleus}{}_{}$}%
    \dimen@\wd\tw@
    \advance\dimen@-\wd\z@
    \divide\dimen@ 3
    \@tempdima\wd\tw@
    \advance\@tempdima-\scriptspace
    \divide\@tempdima 10
    \advance\dimen@-\@tempdima
    \ifdim\dimen@>\z@ \dimen@0pt\fi
    \rel@kern{0.6}\kern-\dimen@
    \if#31
      \overline{\rel@kern{-0.6}\kern\dimen@\macc@nucleus\rel@kern{0.4}\kern\dimen@}%
      \advance\dimen@0.4\dimexpr\macc@kerna
      \let\final@kern#2%
      \ifdim\dimen@<\z@ \let\final@kern1\fi
      \if\final@kern1 \kern-\dimen@\fi
    \else
      \overline{\rel@kern{-0.6}\kern\dimen@#1}%
    \fi
  }%
  \macc@depth\@ne
  \let\math@bgroup\@empty \let\math@egroup\macc@set@skewchar
  \mathsurround\z@ \frozen@everymath{\mathgroup\macc@group\relax}%
  \macc@set@skewchar\relax
  \let\mathaccentV\macc@nested@a
  \if#31
    \macc@nested@a\relax111{#1}%
  \else
    \def\gobble@till@marker##1\endmarker{}%
    \futurelet\first@char\gobble@till@marker#1\endmarker
    \ifcat\noexpand\first@char A\else
      \def\first@char{}%
    \fi
    \macc@nested@a\relax111{\first@char}%
  \fi
  \endgroup
}
\makeatother

\newcommand{\R}{\mathbb {R}}

\renewcommand{\d}{\,\mathrm{d}}

\newcommand{\eps}{\varepsilon}

\usepackage[xcolor]{changebar}
\cbcolor{blue}

\definecolor{addgreen}{rgb}{0,0.45,0}

\definecolor{todoorange}{rgb}{0.85,0.42,0}
\definecolor{addpurple}{rgb}{0.45,0,0.6}
\definecolor{todobrown}{rgb}{0.40,0.22,0.05}

\usepackage[inline]{enumitem}
\newcommand{\enumlabelformat}{\roman}

\newlength{\thelabelsep}
\newcounter{inlineenum}
\renewcommand{\theinlineenum}{\enumlabelformat{inlineenum}}

\let\epsilon\varepsilon
\let\phi\varphi

\newcommand{\krus}{{\mathcal {M}_\mathrm{Krus}}}
\newcommand{\N}{\mathbb{N}}

\newcommand{\nchi}{{\raise.3ex\hbox{$\chi$}}}

\usepackage[runin]{abstract}

\makeatletter
\let\@fnsymbol\@arabic
\makeatother

\allowdisplaybreaks

\usepackage{calc}
\usepackage{scalerel, stackengine}
\stackMath

\newcommand\reallywidehat[1]{%
  \savestack{\tmpbox}{\stretchto{%
    \scaleto{%
      \scalerel*[\widthof{\ensuremath{#1}}]{\kern-.6pt\bigwedge\kern-.6pt}%
      {\rule[-\textheight/2]{1ex}{\textheight}}%
    }{\textheight}%
  }{0.5ex}}%
  \stackon[1pt]{#1}{\tmpbox}%
}

\title{Ideal points, directed completion and the case of the maximally extended Schwarzschild spacetime}

\author{Nicola Gigli\footnotemark[1] \,\orcidlink{0000-0002-5088-2211},
Argam Ohanyan\footnotemark[2] \,\orcidlink{0000-0002-6585-8361},
Marco Picerni\footnotemark[1]\,
\orcidlink{0009-0004-4364-4831}, \\
Zhe-Feng Xu\footnotemark[1]\,\,\textsuperscript{,}\footnotemark[3] \,\orcidlink{0009-0004-7020-8517},
Matteo Zanardini\footnotemark[1] \,\orcidlink{0009-0007-5061-9361}
}
\date{\today}

\begin{document}

\maketitle
\footnotetext[1]{SISSA, 34136, Trieste, Italy, ngigli@sissa.it, mpicerni@sissa.it, zxu@sissa.it, mzanardi@sissa.it} 
\footnotetext[2]{Department of Mathematics, 
  University of Toronto, 
  40 St. George Street, 
  Toronto, Ontario, 
  M5S 2E4, Canada, argam.ohanyan@utoronto.ca} 
\footnotetext[3]{School of Mathematical Sciences, University of Science and Technology of China,  230026, Hefei, China, xzf1998@mail.ustc.edu.cn}

\begin{abstract}
We study the directed completion of Lorentzian pre-length spaces, and we show that, under some natural assumptions which cover the case of smooth globally hyperbolic spacetimes, it coincides with the future causal completion of Geroch--Kronheimer--Penrose. Moreover, we provide applications of our findings by characterizing the directed completion of the Kruskal--Szekeres spacetime in terms of its radial null geodesics.

\vspace{1em}

\vspace{0.2em}

\end{abstract}

\tableofcontents

\section{Introduction}

The notions of boundary and completion of a smooth spacetime are fundamental concepts in Lorentzian Geometry and General Relativity. Various definitions have appeared over the years, starting with Penrose's conformal boundary, introduced by Penrose in \cite{PenroseConfBoundary}, that, among other things, permits to give a mathematical definition of a black hole. Although this construction is geometrically significant, it is not canonical, and it is well-defined only in specific examples.

In 1972, Geroch, Kronheimer, and Penrose \cite{GeKroPe} introduced the definition of causal boundary and causal completion. This more general construction is based on terminal indecomposable past and future sets (TIPs and TIFs) and, loosely speaking, consists of adding some points at infinity to the spacetime which consist of limit points of certain causal curves.
For instance, Budic and Sachs \cite{BudicSachs1974} showed that the causal completion of any causally continuous spacetime has Hausdorff topology. Furthermore, the causal completion was shown to have a natural structure by Harris \cite{Harris_universality}, who proved its universal properties within the category of chronological spaces. Other approaches include the abstract boundary by Scott and Szekeres in \cite{Scott_Szekeres_a_boundary}, the geodesic boundary by Geroch in \cite{Geroch-completion} and Schmidt's bundle boundary \cite{Schmidt1971}.

Motivated by interests in non-smooth Lorentzian geometry, in 2025, the first author introduced in \cite{gigli2025hyperbolicbanachspaces} the notion of directed completion for a general partially ordered set, a setting that includes smooth globally hyperbolic spacetimes (see also \cite{markowsky1976chain, zhao2010dcpo} for the same definition arising from different motivations). This notion is based on the concept of directed sets and it is defined by a universal property in the category of partial orders. Heuristically, the directed completion adds to the original space the suprema of directed sets which do not already possess one. This is a fundamental difference from the causal completion, which adds both past and future limit points.

In this work we show that the future causal completion and the directed completion coincide for smooth globally hyperbolic spacetimes (cf.\ Proposition \ref{prop: case of glob hyp spacetimes}), and more generally, for approximable, past-distinguishing Lorentzian pre-length spaces satisfying the supremum-compatibility condition \eqref{eq:SC}, see Theorem \ref{theorem: equivalence of causal and directed completions}. As the notion of directed completion is more robust than the future causal one (as it only necessitates a partial order), this result corroborates the idea that the former is the appropriate generalization of the latter to non-smooth settings. 
We also point out that this coincidence result is closely related to the universality properties of the future causal completion studied by Harris \cite{Harris_universality} in the setting of chronological sets.

In the last section, we give concrete applications by studying the maximally extended Schwarzschild spacetime. Indeed, we explicitly characterize its directed completion, or equivalently, its future causal completion, by exploiting a case-by-case study of its directed sets using radial null geodesics. This result offers an alternative viewpoint on the construction of the causal completion of the maximally extended Schwarzschild spacetime obtained by Geroch--Kronheimer--Penrose in \cite{GeKroPe}. 

\section{Preliminaries}

\subsection{The directed completion and Lorentzian pre-length spaces}

In this subsection we recall the basics of Lorentzian pre-length spaces and the construction of the directed completion of a partially ordered set. We suggest \cite{gigli2025hyperbolicbanachspaces, KunzingerSamann} for more material on the topic. Let us start with some basic definitions.

\begin{definition}[Directed and lower sets]
    Given a partially ordered set $(\mathrm{X},\leq)$, we say that $D\subseteq \mathrm{X}$ is directed if for any finite $F\subseteq D$ there exists $z\in D$ such that $x\leq z$ for every $x\in F$.
    A lower set $L\subseteq \mathrm{X}$ is a subset such that $x\in L$ and $y\in \mathrm{X}$ with $y\leq x$ implies $y\in L$.
\end{definition}

\begin{definition}[Directed completeness]\label{def:Directedcomplete}
A partially ordered set $(X,\leq)$ is \emph{directed-complete} if every
directed subset $D\subseteq X$ admits a supremum $\sup D$; that is, $ d\leq \sup D$ for every $ d\in D,$
and $\sup D\leq z $ for every upper bound $ z$  of $ D.$
\end{definition}
\begin{definition}
    Given a partially ordered set $(\mathrm X,\leq)$, a subset $A\subseteq \mathrm{X}$ is called directed-sup-closed if for every directed subset $D\subseteq A$ that admits a supremum, one has $\sup D\in A$. Given a subset $A\subseteq \mathrm X$, we denote by $\widehat{A}$ the smallest (with respect to inclusion) lower and directed-sup-closed subset of $\mathrm X$ containing $A$. Finally, a subset $C\subseteq D$ of a directed set $D$ is called \emph{cofinal} if for every $d\in D$ there exists $c\in C$ with $d\leq c$; in this case, $C$ is directed as well, and $C$ and $D$ have the same upper bounds, hence the same supremum whenever one of the two exists.
\end{definition} 

To define the directed completion, we need a notion of natural morphism between directed-complete partial orders, the choice is that of maps with the monotone convergence property.

\begin{definition}[Monotone convergence property, {\cite[Definition~2.2]{gigli2025hyperbolicbanachspaces}}]\label{def:MCP} We say that a map  $T:(\mathrm{X_1},\leq_1)\to (\mathrm{X_2},\leq_2)$
    has the \emph{monotone convergence property}, {\sf Mcp} for short, if for any directed set $D\subseteq \mathrm{X_1}$ admitting a supremum, the set $T(D)\subseteq \mathrm{X_2}$ admits supremum and
    \begin{equation*}\label{eq:MCP}
        \sup T(D)=T(\sup D).
    \end{equation*}
\end{definition}

We can now define the directed completion of a partially ordered set by its universal property following \cite[Definition 2.4]{gigli2025hyperbolicbanachspaces}, see also \cite{markowsky1976chain,zhao2010dcpo}.

\begin{definition}[Directed completion] Let $(\mathrm{X}, \leq)$ be a partially ordered set. A directed completion of $(\mathrm{X}, \leq)$ is given by a directed complete partially ordered set $(\overline{\mathrm{X}}, \overline{\leq})$ and a map $\iota: \mathrm{X} \rightarrow \overline{\mathrm{X}}$ with the {\sf Mcp} that is universal in the following sense: for any directed complete partially ordered set $\left(\mathrm{Z}, \leq_{\mathrm{Z}}\right)$ and map $T: \mathrm{X} \rightarrow \mathrm{Z}$ with the {\sf Mcp} there is a unique map $\bar{T}: \overline{\mathrm{X}} \rightarrow \mathrm{Z}$ with the {\sf Mcp} so that $\bar{T} \circ \iota=T$, i.e., making the following diagram commute
\begin{equation*}
\begin{tikzcd}[column sep=large, row sep=large]
X \arrow[r, "\iota"] \arrow[dr, "T"'] & \bar{X} \arrow[d, "\bar{T}"] \\
& Z
\end{tikzcd}
\end{equation*}
\end{definition}

Finally, let us recall that in \cite[Example 14]{gigli2025hyperbolicbanachspaces} the first author showed that the directed completion of the Minkowski spacetime coincides with the future causal completion of Geroch--Kronheimer--Penrose \cite{GeKroPe}.

The geometric setting in which we will study the directed completion is that of Lorentzian pre-length spaces, introduced by Kunzinger and Sämann in \cite{KunzingerSamann}. Now, we recall the relevant basic definitions and results.

\begin{definition}[Lorentzian pre-length space] \label{def: LPLS}
A quintuple $(\mathrm X,\mathsf d,\ll,\leq,\ell)$ is called a Lorentzian pre-length space if $(\mathrm X,\mathsf d)$ is a separable metric space, $\ll$ is a transitive relation, $\leq$ is a reflexive and transitive relation containing $\ll$ and $\ell:X\times X \to [0,\infty]$ is a lower semicontinuous function such that $\ell(x,y) = 0$ if $x \not \leq y$, $\ell(x,y) > 0$ if and only if $x \ll y$, and for all $x \leq y \leq z$,
\begin{equation}
    \ell(x,z) \geq \ell(x,y) + \ell(y,z).
\end{equation}
\end{definition}

Lorentzian pre-length spaces provide a nonsmooth analogue of time-oriented Lorentzian manifolds: the relations $\leq$ and $\ll$ represent the causal and chronological orders, respectively, while $\ell$ plays the role of the Lorentzian distance.

An alternative framework is given by metric spacetimes \cite{Octet}, where $\ell$ may take the value $-\infty$ and completely determines both $\leq$ and $\ll$. The two approaches are compatible, and the results of Section~\ref{sec:GKP_vs_directed} can be adapted straightforwardly to metric spacetimes. Nevertheless, throughout this work we adopt Definition~\ref{def: LPLS}.

In a Lorentzian pre-length space $\mathrm X$, given a point $x \in \mathrm X$ we define its chronological past by the set $I^-(x)$ and its chronological future by the set $I^+(x)$, where
\[
I^-(x) : = \big\{ y \in \mathrm X : y \ll x \big\} \, , \qquad I^+(x) : = \big\{ y \in \mathrm X : x \ll y \big\} \, .
\]
Analogously, we define its causal past by $J^-(x)$ and its causal future by $J^+(x)$ in terms of the causal relation $\leq$. Moreover, for a subset $A \subseteq X$, we define
\begin{equation*}
    I^\pm(A):=\bigcup_{x \in A} I^\pm(x), \quad J^\pm(A):=\bigcup_{x \in A} J^\pm(x).
\end{equation*}

A $\mathsf d$-locally Lipschitz curve $\gamma: I \to \mathrm X$, where $I \subseteq \mathbb R$ is an interval, is called (future) timelike if $\gamma_s \ll \gamma_t$ holds for all $s < t$, and it is called (future) causal if $\gamma_s \leq \gamma_t$ holds for all $s \leq t$. 

We say that $\mathrm X$ is causally path-connected if every pair of points $x \leq y$ can be connected by a future causal curve and each $x \ll y$ can be connected by a future timelike curve. A causal curve $\gamma:[a,b) \to X$ is called extendible to $[a,b]$ if $\lim_{t \uparrow b}\gamma_t$ exists, otherwise it is called inextendible.

It is easy to infer from the lower semicontinuity of $\ell$ that, if $\mathrm X$ is a Lorentzian pre-length space, then the sets $I^\pm(x) \subset \mathrm{X}$ are open for all $x \in X$, hence so is $I^\pm(A)$ for all $A \subseteq X$. Moreover, if $x \ll y \leq z$ or $x \leq y \ll z$, we have that $x \ll z$. This last property is known as the push-up property, see \cite[Lemma 2.10]{KunzingerSamann}.

In general, Lorentzian pre-length spaces are not partially ordered sets, indeed, $\leq$ need not be antisymmetric.

\begin{definition}[Causality and future/past distinguishing]
A Lorentzian pre-length space is called causal if $\leq$ is antisymmetric. It is called past (resp. future) distinguishing if $I^-(x) = I^-(y)$ (resp. $I^+(x) = I^+(y)$) implies $x = y$.
\end{definition}

It is clear that if $\mathrm X$ is a causal Lorentzian pre-length space, then $(\mathrm X, \leq)$ is a partially ordered set. By the push-up property, past distinction implies that $\leq$ is antisymmetric, indeed any past (or future) distinguishing Lorentzian pre-length space is causal. We conclude this subsection with three auxiliary definitions that will be used in the following section.

\begin{definition}[Approximability]
A Lorentzian pre-length space $\mathrm X$ is called approximable if for each $x \in \mathrm X$, $x \in \overline{I^+(x)} \cap \overline{I^-(x)}$.
\end{definition}

If $\mathrm{X}$ is approximable, then $I^\pm(A)\neq\emptyset$ for every nonempty $A\subseteq\mathrm{X}$. Moreover, $I^\pm(x)\cap U\neq\emptyset$ for every open neighborhood $U$ of $x\in\mathrm{X}$. We also remark that every smooth spacetime is approximable and, in approximable Lorentzian pre-length spaces, $I^-(I^-(A)) = I^-(A)$ holds for all $A \subseteq X$, similarly for $I^+$.

We shall also need the following two properties, which relate the causal structure of $\mathrm X$ to the topology of the underlying metric space. 

\begin{definition}[Closed causal futures and pasts]\label{def: closed causal futures}
A Lorentzian pre-length space $\mathrm X$ is said to have \emph{closed causal futures} if the set $J^+(q)$ is closed in $(\mathrm X, \mathsf d)$ for every $q \in \mathrm X$. Analogously, it is said to have \emph{closed causal pasts} if $J^-(q)$ is closed in $(\mathrm X, \mathsf d)$ for every $q \in X$.
\end{definition}

\begin{definition}[Forward completeness]\label{def: forward complete}
A Lorentzian pre-length space $\mathrm X$ is called \emph{forward complete} if every sequence $(x_n)_{n \in \N} \subseteq \mathrm X$ such that $x_n \ll x_{n+1}$ for every $n \in \N$ and such that $x_n \leq \bar x$ for some $\bar x \in \mathrm X$ and every $n \in \N$ converges in $(\mathrm X, \mathsf d)$.
\end{definition}

\begin{remark}
Forward completeness is a Dedekind-type order completeness condition: every chronologically 
increasing sequence that is bounded above admits a limit. Here, boundedness refers to the order rather than to $\mathsf d$, which is a stronger condition. This a weaker version of the notion of \emph{forward metric spacetime} appearing in \cite[Definition 2.1(iii)]{Octet}. Note that global hyperbolicity (in particular, the compactness of diamonds) implies forward completeness, cf. Proposition \ref{prop: case of glob hyp spacetimes}.
\end{remark}


\subsection{The Schwarzschild spacetime}\label{subsection: schwarzschild}

The Schwarzschild metric was first introduced in \cite{Schwarzschild1916_MassPoint} as a solution to the Einstein vacuum equation
\[
\operatorname{Ric}=0 \, ,
\]
which represents the vacuum region outside a static and spherically symmetric star of mass $M>0$. This metric, with the signature convention $(+, -,-,-)$ and coordinates $(t,r, \theta, \phi)$, is given by 
\begin{equation} \label{eq:Schwarzschild_metric}
    g_{{M}} = \left ( 1-\frac{2M}{r} \right ) \d t^2  - \left ( 1- \frac{2M}{r} \right )^{-1} \d r^2 - r^2 \d \Omega^2 \, ,
\end{equation}
where $\d \Omega^2:=\d \theta ^2 + [ \sin \theta ]^2 \d \phi ^2 $ is the standard round metric on the sphere $\mathbb{S}^2$. That $r = 2M$ is only a coordinate singularity can be seen by considering the Eddington--Finkelstein coordinates \cite{Eddington,Finkelstein}.

The causal structure of the Schwarzschild spacetime is better understood when one considers its maximal analytic extension.
This was achieved independently by Kruskal \cite{Kruskal-paper} and Szekeres \cite{Szekeres1960} through the introduction of Kruskal--Szekeres coordinates $U,V$:
\begin{equation}\label{kruskal UV}
    U:=
\begin{cases}
-e^{-\frac u{4M}}, & r>2M,\\
e^{-\frac u{4M}}, & 0<r<2M,
\end{cases}
\qquad
V:=e^{\frac v{4M}},
\end{equation}
where
\[
u:=t-r-2M\log\left|\frac{r}{2M}-1\right|,\qquad
v:=t+r+2M\log\left|\frac{r}{2M}-1\right|,
\]
which are defined on the so-called maximally extended Schwarzschild spacetime
\[
\krus := \big\{(U,V)\in \mathbb{R}^2 : UV<1\big\}\times \mathbb{S}^2.
\]
This set of $(U,V)$ represents a choice of coordinates different from the $(t,r)$ used in \eqref{eq:Schwarzschild_metric}, and the two are related via
\begin{equation}\label{relation uv}
-UV=\left(\frac{r}{2M}-1\right)e^{\frac r{2M}},
\end{equation}
In the $(U,V)$-coordinates, the metric extends smoothly across \(r=2M\), which is now characterized by $UV=0$, and takes the form
\begin{equation}\label{Kruskal}
g_M=\frac{32M^3}{r}e^{-\frac{r}{2M}}\d U\d V-r^2\d \Omega^2,    
\end{equation}
and the vector field $\partial_U+\partial_V$ defines the time orientation (cf. \cite{sbierski}).
In particular, note that the map
\begin{equation}\label{eq:symmetric_map_UV}
    \mathcal{T}:(U,V,\Omega)\to (V,U,\Omega)
\end{equation}
is a time-orientation-preserving isometry of $\krus$, and that the coordinates $U$ and $V$ are monotonically increasing along every future directed causal curve. Moreover, $\krus$ is a globally hyperbolic spacetime \cite{GerochDomDep}. We can partition $\krus$ into five regions:
\[
\begin{array}{ll}
\text{i)  right exterior region } \mathrm{I}:=\{U<0,\ V>0\}, &
\text{ii)  black hole  } \mathrm{II}:=\{U>0,\ V>0\}, \\[0.3em]
\text{iii)
white hole } \mathrm{III}:=\{U<0,\ V<0\}, &
\text{iv)  left exterior region } \mathrm{IV}:=\{U>0,\ V<0\}, \\[0.3em]
\multicolumn{2}{l}{\text{v) horizons } \mathcal{H}:= \{r=2M\}=\{UV=0\}.}
\end{array}
\]
The Penrose diagram \cite{carter1966complete} of $\krus$ is given by
\begin{figure}[H]
    \centering
    \includegraphics[scale=0.42]{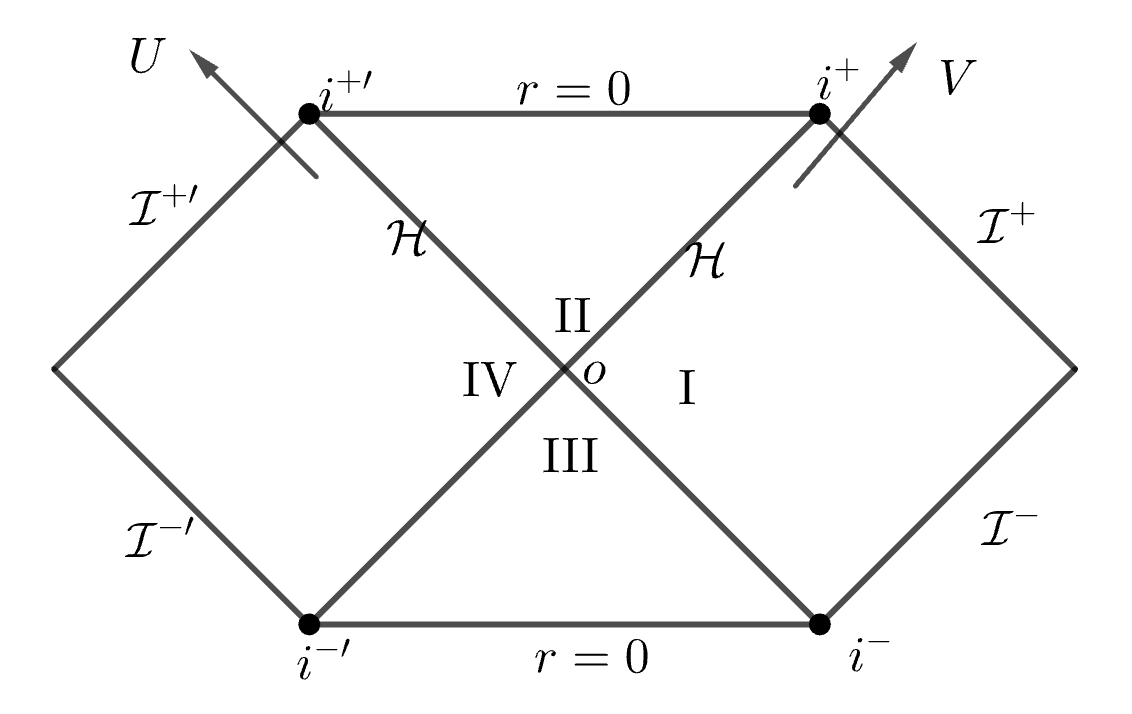}
    \caption{The Penrose diagram of $(\krus, g_{\krus})$.}
    \label{fig:penrose}
\end{figure}

In Figure \ref{fig:penrose}, we have denoted by $\mathcal{I}^+,\mathcal{I}^{+'}$ the future null infinity, $\mathcal{I}^-,\mathcal{I}^{-'}$ the past null infinity, $i^+, i^{+'}$ the future timelike infinity, and $i^-, i^{-'}$ the past timelike infinity while \(r=0\) corresponds to the Schwarzschild singularity. 
Moreover, Figure \ref{fig:penrose} was used to geometrically describe the causal completion of $\krus$ in \cite[Figure 3]{GeKroPe}.

We conclude this section by recalling Sbierski's result \cite{sbierski} that $\krus$ is $C^0$-inextendible. While its $C^2$-inextendibility follows from the blow-up of the Kretschmann scalar, this establishes inextendibility even at lower regularity.

\section{The compatibility of the notions of completion} \label{sec:GKP_vs_directed}

\subsection{The future causal completion}

In this subsection, we adapt the construction of the future causal completion, introduced by Geroch, Kronheimer, and Penrose in \cite{GeKroPe}, to the setting of Lorentzian pre-length spaces. We refer the reader to the recent works \cite{AkeBurSol, BurFloHer} for more information about the topic.

Let $\mathrm X$ be a Lorentzian pre-length space. A subset $U \subseteq X$ is called past if $I^-(U) = U$. Recall that past sets are necessarily open. 

\begin{lemma}
\label{lemma: one inclusion for past sets}
Let $\mathrm X$ be an approximable Lorentzian pre-length space. An open subset $U \subseteq \mathrm X$ is past if and only if $I^-(U) \subseteq U$.
\end{lemma}
\begin{proof}
    We show that $U \subseteq I^-(U)$ for any open set $U$. Indeed, let $x \in U$, by approximability $U \cap I^+(x) \neq \emptyset$, and for every $y \in U \cap I^+(x)$, we have that $x \in I^-(y) \subseteq I^-(U)$.
\end{proof}

In the following definition we introduce the indecomposable past sets.

\begin{definition}[Indecomposable past sets]
A nonempty past set $P \subseteq \mathrm X$ is called an \emph{indecomposable past set} if whenever $P = U_1 \cup U_2$ for past sets $U_1,U_2 \subseteq P$, then $P = U_1$ or $P = U_2$.
\end{definition}

\begin{lemma}[Pasts of points are indecomposable]
\label{lem: I^-(x) indecomp}
Let $\mathrm X$ be an approximable Lorentzian pre-length space. Then for each $x \in \mathrm X$, $I^-(x)$ is an indecomposable past set.
\end{lemma}
\begin{proof}
    The set $I^-(x)$ is trivially a past set, we show that it is indecomposable. Indeed, suppose $I^-(x) = U \cup V$ with nonempty past proper subsets $U, V \subsetneq I^-(x)$ and let $y \in U \setminus V$, $z \in V \setminus U$. By approximability, there is $w \in I^-(x) \cap I^+(y) \cap I^+(z)$. Since $I^-(x) = U \cup V$, either $w \in U$ or $w \in V$. If $w \in U$, then $z \in I^-(w) \subseteq I^-(U) = U$, which is a contradiction. A similar contradiction is concluded assuming $w \in V$.
\end{proof}

Following \cite{GeKroPe} (and Lemma \ref{lem: I^-(x) indecomp}), we distinguish two kinds of indecomposable past sets.

\begin{definition}[Proper and terminal indecomposable past sets]\label{def: PIP TIP}
An indecomposable past set $P \subseteq \mathrm X$ is called \emph{proper}, and we call $P$ a PIP, if $P = I^-(x)$ for some $x \in \mathrm X$; it is called \emph{terminal}, and we call $P$  a TIP, otherwise.
\end{definition}

Here, \emph{proper} does not mean proper as a subset; rather, it indicates that \(P=I^-(x)\) for some \(x\in\mathrm X\), as opposed to an ideal point. Moreover, TIPs are unrelated to the \emph{tip of a set} introduced in \cite[Definition 2.1]{gigli2025hyperbolicbanachspaces}, which will not be used here.

In the next lemma we show a directedness property of indecomposable past sets which furnishes the key link between the future causal and directed completions.

\begin{lemma}[Directedness of indecomposable past sets]
\label{lem: directedness of IP}
Let $\mathrm X$ be an approximable Lorentzian pre-length space and $P \subseteq \mathrm X$ an indecomposable past set. Then for any $x,y \in P$ there exists $z \in P$ such that $x,y \in I^-(z)$.
\end{lemma}
\begin{proof}
    Suppose the claim is false, so there exist $x,y \in P$ such that
    \begin{equation} \label{eq: constradiction assumption}
        I^+(x) \cap I^+(y) \cap P = \emptyset.
    \end{equation}
    We define the past sets
    \[
    U:=I^-\big(P \setminus I^+(x)\big) \, ,  \qquad V:=I^-\big(P \setminus I^+(y)\big) \, 
    \]
    and we claim that $P = U \cup V$.
    
    It is clear that $U \cup V \subseteq P$. Conversely, let $p \in P$. Then, since $\mathrm X$ is approximable and $P$ is open, there exists $q \in P$ with $p \ll q$ and so, by \eqref{eq: constradiction assumption}, we have that $q \notin I^+(x) \cap I^+(y)$. If $q \notin I^+(x)$, then $q \in P \setminus I^+(x)$, hence $p \in I^-(q) \subseteq I^-(P \setminus I^+(x)) = U$. If $q \notin I^+(y)$, analogously we get that $p \in V$. In particular, we conclude $P = U \cup V$.
    
    Notice that $x \notin U$; indeed if that were not the case there would exist $q \in P \setminus I^+(x)$ with $x \ll q$, i.e., $q \in I^+(x)$, a contradiction. Similarly, $y \notin V$. Thus, $P=U\cup V$ decomposes $P$ into two proper past subsets, contradicting its indecomposability.
\end{proof}

In the following definition we introduce the past chronological inclusion relation.

\begin{definition}[Past chronological inclusion relation]
Let $\mathrm X$ be a Lorentzian pre-length space. The \emph{past chronological inclusion relation} $\preceq_{I^-}$ is defined by 
\[
x\preceq_{I^-}y
\quad\Longleftrightarrow\quad
I^-(x)\subseteq I^-(y).
\]
\end{definition}

\begin{remark}
  The relation $\preceq_{I^-}$ is always reflexive and transitive, and it is antisymmetric---hence a partial order---if and only if $\mathrm X$ is past distinguishing. In this case, Lemma~\ref{lem: directedness of IP} states that every indecomposable past set $P\subseteq\mathrm X$ is $\preceq_{I^-}$-directed. 
\end{remark}

\begin{lemma}[Upper bounds and unions of chronological pasts]\label{lem: upper bounds and unions}
Let $\mathrm X$ be a past distinguishing and approximable Lorentzian pre-length space and let $A \subseteq \mathrm X$ be nonempty. Then, for every $y \in \mathrm X$, the point $y$ is a $\preceq_{I^-}$-upper bound of $A$ if and only if $I^-(A) \subseteq I^-(y)$. Moreover, $A$ and the past set $I^-(A)$ have the same $\preceq_{I^-}$-upper bounds. In particular, $A$ admits a $\preceq_{I^-}$-supremum if and only if $I^-(A)$ does, and in this case the two suprema coincide.
\end{lemma}
\begin{proof}
The first assertion is a restatement of the definition of $\preceq_{I^-}$: we have $a \preceq_{I^-} y$ for every $a \in A$ if and only if $I^-(a) \subseteq I^-(y)$ for every $a \in A$, i.e., if and only if $I^-(A) \subseteq I^-(y)$. Furthermore, the set $I^-(A)$ is clearly a past set. Applying the first assertion to $A$ and to $I^-(A)$ we conclude that the two sets have the same upper bounds; since suprema are least upper bounds, the last assertion follows.
\end{proof}

\begin{lemma}[$\preceq_{I^-}$ vs.\ $\leq$]
\label{lemma: comparing chron inclusion and causal relation}
Let $\mathrm X$ be a past distinguishing and approximable Lorentzian pre-length space with closed causal pasts. Then the past chronological inclusion relation $\preceq_{I^-}$ agrees with the causal relation $\leq$.
\end{lemma}
\begin{proof}
    Suppose that $x \leq y$. Then the push-up property immediately implies $x \preceq_{I^-} y$. Conversely, suppose $I^-(x) \subseteq I^-(y)$, hence we get that $\overline{I^-(x)} \subseteq \overline{I^-(y)}$. By approximability, $x \in \overline{I^-(x)}$ and so, since $\overline{I^-(y)} \subseteq J^-(y)$, which follows by the assumed closedness of $J^-(y)$, the conclusion follows.
\end{proof}

The following proposition characterizes indecomposable past sets in the setting of  approximable Lorentzian pre-length spaces with closed causal futures. The result is inspired by \cite[Theorem 2.1]{GeKroPe}.

\begin{proposition}[Description of indecomposable past sets]\label{prop:description of IP}
Let $\mathrm X$ be an approximable Lorentzian pre-length space with closed causal futures and $P \subseteq \mathrm X$ an indecomposable past set. Then there exists a sequence $(x_n)_{n \in \N} \subseteq P$ with $x_n \ll x_{n+1}$ for every $n \in \N$ and such that $P = \bigcup_{n=1}^\infty I^-(x_n)$. Moreover, for any such sequence precisely one of the following holds:
\begin{enumerate}
\item[i)] The sequence $(x_n)_{n \in \N}$ admits a subsequence converging to some $x \in \mathrm X$, then $P = I^-(x)$, so that $P$ is a PIP; if, in addition, $\mathrm X$ is past distinguishing, then every convergent
subsequence of $(x_n)_{n\in\mathbb N}$ has the same limit $x$.
\item[ii)] The sequence $(x_n)_{n \in \N}$ admits no convergent subsequence and $P = I^-(\mathsf c)$ with $\mathsf c := \{x_n: n \in \mathbb N \}$.
\end{enumerate}
 In case $\mathrm X$ is causally path-connected, $\mathsf c$ may be replaced by a future inextendible timelike curve $\gamma$, so that $P = I^-(x)$ or $P = I^-(\gamma)$. Moreover, if no sequence $(x_n)$ as above admits a convergent subsequence, then $P \neq I^-(x)$ for any $x \in \mathrm X$, i.e., $P$ is a TIP.
\end{proposition}

\begin{proof}
    Let $(y_n)_{n \in \N}$ be countable and dense in $P$.
    We claim that $P = \bigcup_{n=1}^\infty I^-(y_n)$. Indeed, let $p \in P$, by approximability there exists $q \in P$ with $p \ll q$. Then the set $I^+(p) \cap P$ is an open neighborhood of $q$, so $y_n \in I^+(p) \cap P$ for some $n$, so that $p \in I^-(y_n)$. 
    
    We now construct a sequence $(x_n)_{n \in \N}$ such that $x_n \ll x_{n+1}$ and $y_n \ll x_n$ for all $n$, in which case we automatically have $P = \bigcup_{n = 1}^\infty I^-(x_n)$.  Since $P$ is open and $\mathrm X$ is approximable, there exists $x_1 \in P$ with $y_1 \ll x_1$. By Lemma \ref{lem: directedness of IP}, there exists $x_2 \in P$ such that $x_1,y_2 \ll x_2$. Iteratively, we obtain a sequence $(x_n)$ with the claimed properties.

    We prove the dichotomy. Assume that $x_{n_k} \to x$ for some subsequence and some $x \in \mathrm X$. For fixed $n \in \N$ we have $x_{n_k} \in J^+(x_n)$ for every $k$ large enough, and $J^+(x_n)$ is closed by assumption, whence $x_n \leq x$. Since $x_n \ll x_{n+1} \leq x$, the push-up property gives $x_n \ll x$ for every $n$, so that $P = \bigcup_{n=1}^\infty I^-(x_n) \subseteq I^-(x)$. Conversely, let $p \ll x$: the set $I^+(p)$ is open and contains $x$, hence it contains $x_{n_k}$ for $k$ large enough, i.e., $p \in I^-(x_{n_k}) \subseteq P$. Therefore $P = I^-(x)$. If $\mathrm X$ is past distinguishing and $x'$ is the limit of another
convergent subsequence, the same argument gives $I^-(x')=P=I^-(x),$ and hence $x'=x$.

    If instead $(x_n)_{n \in \N}$ admits no convergent subsequence, then, $P = I^-(\mathsf c)$ with $\mathsf c := \{x_n : n \in \N\}$ by construction.
    
    Finally, if $X$ is causally path-connected, one may connect $x_n$ to $x_{n+1}$ with a timelike curve $\gamma_n$ and let $\gamma$ be the concatenation of the $\gamma_n$, so that $P = I^-(\gamma)$.
\end{proof}

We conclude the subsection with the definition of future causal completion. 

\begin{definition}[Future causal completion] \label{def: Future causal completion}
Let $\mathrm X$ be a Lorentzian pre-length space. We define
\begin{equation}
    \mathcal C^+(\mathrm X):=\big\{P \subseteq \mathrm X : P \emph{\text{ is an indecomposable past set}}\big\} \, .
\end{equation}
We say that the partially ordered set $(\mathcal C^+(\mathrm X), \subseteq)$, where $\subseteq$ is the inclusion of subsets of $\mathrm X$, is the future causal completion of $\mathrm X$.
\end{definition}

\subsection{Proof of the equivalence}\label{subsection: proof of the equivalence}

In this subsection, we prove the main result of this section, namely that the future causal completion, in the sense of Definition \ref{def: Future causal completion} of an approximable and past distinguishing Lorentzian pre-length space satisfying the additional assumption \eqref{eq:SC} (abbreviating supremum-compatibility) below coincides with its directed completion. We start with a preliminary lemma that is closely related to \cite[Theorem 5 (4)]{Harris_universality}.

\begin{lemma}[Causal completion is directed complete]
\label{lemma: causal boundary directed complete}
Let $\mathrm X$ be an approximable Lorentzian pre-length space. Then $(\mathcal C^+(\mathrm X),\subseteq)$ is a directed complete partially ordered set. Moreover, the supremum of a directed family $\mathcal D \subseteq \mathcal C^+(\mathrm X)$ is given by $\bigcup_{P \in \mathcal D} P$, which is, in particular, an indecomposable past set.
\end{lemma}

\begin{proof}
    It is clear that $(\mathcal C^+(\mathrm X),\subseteq)$ is a partially ordered set. Now we show that it is directed complete. Let $ \mathcal D \subseteq \mathcal C^+(\mathrm X)$ be $\subseteq$-directed. We claim that $\mathcal D$ admits a $\subseteq$-supremum. Let us write $\mathcal D = \{ P_i: i \in I\}$ for convenience, where $I$ is some index set. We claim $ P:=\bigcup_{i \in I} P_i$
    is the supremum of $\mathcal D$. 
    
    It is easy to check that $P$ is a past set. Moreover, suppose by contradiction that $P$ is decomposable, so that $P = U \cup V$ with proper past subsets $U,V \subseteq P$. Then
    \begin{equation}
        P = U \cup V = \bigcup_{i \in I} \, (P_i \cap U) \cup (P_i \cap V).
    \end{equation}
    Suppose $x,y \in P$ with $x \in U \setminus V$, $y \in V \setminus U$. Since $D$ is directed, there exists some $k \in I$ such that $x,y \in P_k$, thus $P_k = (P_k \cap U) \cup (P_k \cap V)$ is a decomposition of $P_k$ into proper subsets. But now $P_k \cap U$ and $P_k \cap V$ are past sets, indeed 
    \[
    I^-(P_k \cap U) \subseteq I^-(P_k) \cap I^-(U) = P_k \cap U \, .
    \]
    Since $X$ is approximable and $P_k \cap U$ is open, Lemma \ref{lemma: one inclusion for past sets} implies that $P_k \cap U$ is a past set. This is a contradiction, since $P_k \in \mathcal C^+(\mathrm X)$ is indecomposable past.
    
    Thus, $P \in \mathcal C^+(\mathrm X)$ and clearly it is the supremum of $\mathcal D$.
\end{proof}

By the very definition of directed completion, the map $\iota$ is required to have the {\sf Mcp}. By Lemma \ref{lemma: causal boundary directed complete}, for every $\preceq_{I^-}$-directed $D \subseteq \mathrm X$ the $\preceq_{I^-}$-supremum of $\iota(D)$ in $\mathcal C^+(\mathrm X)$ is $I^-(D)$; hence $\iota$ has the {\sf Mcp} if and only if
\begin{equation}\label{eq:SC}\tag{SC}
I^-(\sup D) = I^-(D) \qquad \text{for every $\preceq_{I^-}$-directed } D \subseteq \mathrm X \text{ admitting a $\preceq_{I^-}$-supremum}.
\end{equation}
This is an assumption on $\mathrm X$, and not a consequence of approximability and past distinction. We emphasize that \eqref{eq:SC} is a new notion of completeness that does not follow from the property $\sup I^-(x) = x$ (the $\sup$ is of course with respect to $\preceq_{I^-}$), which holds in every approximable and past distinguishing Lorentzian pre-length space (cf.\ the proof of Theorem \ref{theorem: equivalence of causal and directed completions} below), but is implied by forward completeness (in particular, by global hyperbolicity), cf.\ Proposition \ref{prop: forward implies SC}.

The next lemma identifies the geometric meaning of \eqref{eq:SC}: It holds precisely when no ideal point of $\mathrm X$ is captured by a supremum.

\begin{lemma}[Reformulation of \eqref{eq:SC}]\label{lem: SC criterion}
Let $\mathrm X$ be an approximable and past distinguishing Lorentzian pre-length space, let $D \subseteq \mathrm X$ be $\preceq_{I^-}$-directed and admitting a $\preceq_{I^-}$-supremum $x := \sup D$. Then
\begin{enumerate}
\item[i)] $I^-(D) \subseteq I^-(x)$;
\item[ii)] $I^-(D) = I^-(x)$ if and only if $I^-(D) = I^-(y)$ for some $y \in \mathrm X$.
\end{enumerate}
Consequently, $\mathrm X$ satisfies \eqref{eq:SC} if and only if no TIP admits a $\preceq_{I^-}$-supremum in $\mathrm X$.
\end{lemma}
\begin{proof}
i) For every $d \in D$, we have $d \preceq_{I^-} x$, i.e., $I^-(d) \subseteq I^-(x)$, and the claim follows.

ii) One implication is trivial, since $I^-(x)$ is the chronological past of a point. Conversely, assume $I^-(D) = I^-(y)$. By Lemma \ref{lem: upper bounds and unions}, $y$ is a $\preceq_{I^-}$-upper bound of $D$, and hence $x \preceq_{I^-} y$ by the definition of supremum, that is, $I^-(x) \subseteq I^-(y) = I^-(D)$. Together with i), this gives $I^-(D) = I^-(x)$.

For the last assertion, assume first that some TIP $P$ admits a supremum. By Lemma \ref{lem: directedness of IP}, $P$ is a directed subset of $(\mathrm X, \preceq_{I^-})$, and $I^-(P) = P$ because $P$ is a past set; since $P$ is not the chronological past of any point, ii) shows that $P = I^-(P) \neq I^-(\sup P)$, so that \eqref{eq:SC} fails for $D := P$. Conversely, assume that \eqref{eq:SC} fails for some directed $D$ with supremum $x$. Then $I^-(D) \neq I^-(x)$, and $I^-(D)$ is an indecomposable past set by Lemmas \ref{lem: I^-(x) indecomp} and \ref{lemma: causal boundary directed complete}; by ii) it is not the chronological past of any point, i.e., it is a TIP. By Lemma \ref{lem: upper bounds and unions}, $D$ and $I^-(D)$ have the same suprema, so $I^-(D)$ is a TIP admitting the supremum $x$.
\end{proof}

\begin{remark}
Lemma \ref{lem: SC criterion} makes the failure of \eqref{eq:SC} transparent: it can only occur when a terminal indecomposable past set, i.e., an ideal point of $\mathrm X$, admits a \emph{least} upper bound in $\mathrm X$. Note that this is a strictly stronger requirement than the existence of upper bounds, which is a common occurrence in spacetimes with holes or naked singularities.
\end{remark}

The following theorem is the main result of this section. We show that, under some assumptions, the directed completion of a Lorentzian pre-length space coincides with the future causal completion. 

\begin{theorem}[Equivalence of causal and directed completions]
\label{theorem: equivalence of causal and directed completions}
Let $\mathrm X$ be an approximable and past distinguishing Lorentzian pre-length space satisfying \eqref{eq:SC} and define the map $\iota: \mathrm X \to \mathcal C^+(\mathrm X)$ by $x \mapsto \iota (x) : = I^-(x)$. Then, $( \mathcal C^+ (\mathrm X), \subseteq)$ is the directed completion of $(\mathrm X,\preceq_{I^-})$. 

Moreover, $(\mathcal C^+(\mathrm X),\subseteq)$ is the directed completion of $(\mathrm X,\leq)$ provided $\mathrm X$ has closed causal pasts.
\end{theorem}
\begin{proof}
   First of all, $\iota$ has the {\sf Mcp}: if $D \subseteq \mathrm X$ is $\preceq_{I^-}$-directed and admits a $\preceq_{I^-}$-supremum, then $\iota(D)$ is $\subseteq$-directed in $\mathcal C^+(\mathrm X)$ because $\iota$ is monotone, its $\subseteq$-supremum is $I^-(D)$ by Lemma \ref{lemma: causal boundary directed complete}, and this set equals $I^-(\sup D) = \iota(\sup D)$ by \eqref{eq:SC}.

   Let $(Z,\leq_Z)$ be a directed complete partially ordered set. Let $T:(\mathrm X,\preceq_{I^-}) \to (Z,\leq_Z)$ be a map satisfying the {\sf Mcp}. Notice that $T$ is monotone: if $x \preceq_{I^-} y$, then $D := \{x,y\}$ is directed with $\sup D = y$, and the {\sf Mcp} gives $T(y) = \sup\{T(x),T(y)\} \geq_Z T(x)$. We define the map $\bar T:(\mathcal C^+(\mathrm X), \subseteq) \to (Z,\leq_Z)$ as follows:
   \begin{equation} \label{eq:3.5}
       \bar T(P):=\sup \left \{T(x) : x \in P \right \}.
   \end{equation}
   This map is well-defined. Indeed, if $P \in \mathcal C^+(\mathrm X)$, then $P \subseteq X$ is $\preceq_{I^-}$-directed by Lemma \ref{lem: directedness of IP}. Thus, since $T$ is monotone, $T(P) \subseteq Z$ is $\leq_Z$-directed and so $T(P)$ admits a supremum since $(Z,\leq_Z)$ is directed complete. Let us stress that $P$ itself need not admit a $\preceq_{I^-}$-supremum in $\mathrm X$, so that the {\sf Mcp} of $T$ cannot be invoked directly here.

   We claim that $\bar T \circ \iota =T$. First of all, notice that $I^-(x) \subseteq X$ is indecomposable past (cf.\ Lemma \ref{lem: I^-(x) indecomp}) and thus $\preceq_{I^-}$-directed by Lemma \ref{lem: directedness of IP}, and it is an elementary consequence of approximability that $\sup I^-(x) = x$. Now, using the {\sf Mcp} of $T$, we obtain that
   \begin{equation}
       \bar T \circ \iota(x) = \bar T(I^-(x)) = \sup \left \{T(y) : y \in I^-(x) \right \} = T (\sup I^-(x) ) = T(x).
   \end{equation}
   Now, we show that $\bar T$ has the {\sf Mcp}. To this end, let $\mathcal D \subseteq \mathcal C^+(\mathrm X)$ be $\subseteq$-directed. By Lemma \ref{lemma: causal boundary directed complete}, $\sup \mathcal D$ is well-defined and coincides with $\bigcup_{P \in \mathcal D} P$. Hence
   \begin{equation}
       \bar T( \sup \mathcal D)= \sup \big \{ \sup  \{T(x) : x \in P  \} : P \in \mathcal D\big \} = \sup \left \{ \bar T(P) : P \in \mathcal D \right \} \, . 
   \end{equation}
   The interchange of the two suprema in the last identity is justified by the fact that, for a directed family of directed subsets of $(Z,\leq_Z)$, the supremum of the suprema and the supremum of the union coincide, both being characterized by the same set of upper bounds.

   Finally, the uniqueness of $\bar T$ is easy to establish. It follows from the fact that any past set $P$ is equal to $\bigcup_{x \in P} I^-(x)$ and that, by Lemma \ref{lem: directedness of IP}, the family $\{\iota(x) : x \in P\}$ is $\subseteq$-directed with $\subseteq$-supremum $P$ in $\mathcal C^+(\mathrm X)$, so any possible extension map with the {\sf Mcp} needs to satisfy \eqref{eq:3.5}. In particular, $(\mathcal C^+( \mathrm X),\subseteq)$ is the directed completion of $(\mathrm X, \preceq_{I^-})$.

   Moreover, the last claim concerning the causal relation follows from the compatibility in Lemma \ref{lemma: comparing chron inclusion and causal relation}.
\end{proof}

In the above context, the set $\mathcal C^+(\mathrm X) \setminus \iota(\mathrm X)$ is the \emph{future causal boundary} of $\mathrm X$. We conclude this section by exhibiting a checkable sufficient condition for \eqref{eq:SC}, and we verify that it holds in the most relevant class of examples, namely, globally hyperbolic (smooth) spacetimes.

\begin{proposition}[Forward completeness implies \eqref{eq:SC}]\label{prop: forward implies SC}
Let $\mathrm X$ be an approximable and past distinguishing Lorentzian pre-length space with closed causal futures. If $\mathrm X$ is forward complete, then it satisfies \eqref{eq:SC}.
\end{proposition}
\begin{proof}
Let $D \subseteq \mathrm X$ be $\preceq_{I^-}$-directed and admitting a $\preceq_{I^-}$-supremum $x$. By Lemma \ref{lemma: causal boundary directed complete}, $I^-(D)$ is an indecomposable past set, so that Proposition \ref{prop:description of IP} provides a sequence $(x_n)_{n \in \N} \subseteq I^-(D)$ with $x_n \ll x_{n+1}$ and $I^-(D) = \bigcup_{n=1}^\infty I^-(x_n)$. By Lemma \ref{lem: SC criterion} i) we have $I^-(D) \subseteq I^-(x)$, hence $x_n \leq x$ for every $n$. By forward completeness, $(x_n)_{n \in \N}$ converges, so that Proposition \ref{prop:description of IP} i) gives $I^-(D) = I^-(w)$ with $w := \lim_n x_n$. Lemma \ref{lem: SC criterion} ii) then yields $I^-(D) = I^-(x)$, which is \eqref{eq:SC}.
\end{proof}

\begin{proposition}[The case of globally hyperbolic spacetimes]
\label{prop: case of glob hyp spacetimes}
Let $(\mathcal M,g)$ be a smooth globally hyperbolic spacetime. Then the directed completion of $(\mathcal M,\leq)$ is $(\mathcal C^+(\mathcal M),\subseteq)$.
\end{proposition}
\begin{proof}
    This follows from our main Theorem \ref{theorem: equivalence of causal and directed completions} in conjunction with Proposition \ref{prop: forward implies SC} upon noticing that if $\mathsf d$ is the Riemannian distance of an arbitrary Riemannian metric on $\mathcal M$, then $(\mathcal M,\sf d, \ll, \leq, \ell)$ is an approximable, past-distinguishing Lorentzian pre-length space with closed causal futures and pasts, and it is easily seen to be forward complete as a consequence of global hyperbolicity (cf.\ \cite[Rem.\ 2.10]{Octet} and also \cite[Thm.\ A.4]{ohanyansalamocandal}).
\end{proof}

\begin{remark}
\begin{enumerate}
\item[]
\item The two assumptions of Proposition \ref{prop: forward implies SC} play different roles. Closed causal futures are used only to pass from convergence of the sequence $(x_n)_{n \in \N}$ to the identity $I^-(D) = I^-(\lim_n x_n)$, i.e., to guarantee that the chronological past does not jump at the limit point. Forward completeness is the substantial assumption, and it is exactly what fails in the presence of an ideal point admitting a least upper bound.
\item Proposition \ref{prop: case of glob hyp spacetimes} holds in fact much more generally in the class of globally hyperbolic Lorentzian length spaces (cf.\ \cite{KunzingerSamann}), which includes globally hyperbolic Finsler spacetimes.
\end{enumerate}
\end{remark}

\section{The completion of Schwarzschild spacetime}\label{sec:Schwarzschild}

\subsection{The main result}

In this section, we study the directed completion of $\krus$ in terms of its radial null geodesics. We recall that, as a consequence of Theorem \ref{theorem: equivalence of causal and directed completions}, this characterization also provides a description of its future causal completion, thereby offering an alternative viewpoint on the description obtained by Geroch--Kronheimer--Penrose in \cite{GeKroPe}. Recall that $\krus$ is globally hyperbolic \cite{GerochDomDep}, thus by Proposition \ref{prop: case of glob hyp spacetimes} it falls within the scope of Theorem \ref{theorem: equivalence of causal and directed completions}, so that its directed completion and its future causal completion may indeed be identified.

Throughout this section, we denote by $\leq$ the causal order in $\krus$. Moreover, we will use the following future directed null curves $\gamma^{U,\Omega},\gamma^{V,\Omega}$ defined by
\begin{equation}\label{eq:def:radial_null_geodesics}
    \gamma^{U,\Omega}(s):=(U,s,\Omega),\qquad \gamma^{V,\Omega}(s):=(s,V,\Omega),
\end{equation}
where $\Omega\in \mathbb{S}^2$ and $U,V\in\R$ are fixed. 
The domain of these curves is considered in such a way that these are inextendible. Moreover, up to a suitable reparameterization, these are radial null geodesics that are globally maximizing, hence globally achronal.

We define the following objects:
\begin{itemize}
    \item The black hole singularity, consisting of $S^+ := (0,+\infty) \times \mathbb{S}^2$, whose generic point is denoted by $[U^*, \Omega^*]$ for some $U^* \in (0,+\infty)$ and $\Omega^* \in \mathbb{S}^2$.
    \item The future timelike infinity, consisting of two points $i^+$ and $i^{+'}$.
    \item The future null infinity, consisting of  the two disjoint tagged copies
    \[
    \mathcal{I}^{+'} := (-\infty,0) \times \mathbb{S}^2\times\{\mathrm L\},
    \qquad
    \mathcal{I}^+ := (-\infty,0) \times \mathbb{S}^2\times\{\mathrm R\} ,
    \]
    where the tags $\mathrm R,\mathrm L$ are introduced only to keep the two copies distinct in the disjoint union below and will be omitted from the notation. The generic point of $\mathcal I^+$ is denoted by $[U^*, \Omega^*]\in \mathcal{I}^+$ for some $U^* \in (-\infty,0)$ and $\Omega^* \in \mathbb{S}^2$, while the generic point of $\mathcal I^{+'}$ is denoted by $[V^*, \Omega^*]\in \mathcal{I}^{+'}$ for some $V^* \in (-\infty,0)$ and $\Omega^* \in \mathbb{S}^2$.
\end{itemize}
We now state the main theorem of this section.
\begin{theorem}\label{main theorem}
    The directed completion of $\krus$ is isomorphic to $(\overline{\mathcal{M}}_{\mathrm{Krus}}, \preceq)$, where
    \[
    \overline{\mathcal{M}}_{\mathrm{Krus}}:= \krus \sqcup S^+ \sqcup \{ i^+\} \sqcup \{ i^{+'}\}\sqcup \mathcal{I}^+\sqcup \mathcal{I}^{+'},
    \]
    and $\preceq$ is defined as follows:
\begin{itemize}
    \item For any $p, q\in \krus$,
    \begin{equation}\label{order1}
        p \leq q \quad \iff\quad  p \preceq q.
    \end{equation}
    \item For any $p \in \{(U,V) : U \in (-\infty,0], V \in \mathbb{R},\, UV <1\}\times \mathbb{S}^2 \cup \mathcal{I}^{+}$,
    \begin{equation}\label{order2}
        p \preceq i^+,\qquad\text{and moreover } i^+\preceq i^+.
    \end{equation}
      \item For any $p \in \{(U,V) : U\in \mathbb{R}, V \in (-\infty,0],\, UV<1\}\times \mathbb{S}^2 \cup \mathcal{I}^{+'}$,
    \begin{equation}\label{order3}
        p \preceq i^{+'},\qquad\text{and moreover } i^{+'}\preceq i^{+'}.
    \end{equation}
    \item Let $[U^*, \Omega^*], [U^*_1, \Omega^*_1], [U^*_2, \Omega^*_2] \in \mathcal{I}^+$ and let $p \in \mathcal M_{\mathrm{Krus}}$. Then 
    \begin{equation}\label{order4}
    \begin{aligned}
        [U^*_1, \Omega^*_1] \preceq [U^*_2, \Omega^*_2]\quad &\iff \quad \Omega^*_1 = \Omega^*_2 \quad\mathrm{and}\quad U^*_1 \leq U^*_2,\\
         p \preceq [U^*, \Omega^*]\quad &\iff \quad p \in J^-\left( \gamma^{U^*,\Omega^* }\right) ,
    \end{aligned}
    \end{equation} 
     \item Let $[V^*, \Omega^*], [V^*_1, \Omega^*_1], [V^*_2, \Omega^*_2] \in \mathcal{I}^{+'}$ and let $p \in \mathcal M_{\mathrm{Krus}}$. Then 
    \begin{equation}\label{order5}
    \begin{aligned}
        [V^*_1, \Omega^*_1] \preceq [V^*_2, \Omega^*_2]\quad &\iff \quad \Omega^*_1 = \Omega^*_2 \quad\mathrm{and}\quad V^*_1 \leq V^*_2,\\
         p \preceq [V^*, \Omega^*]\quad &\iff \quad p \in J^-\left( \gamma^{V^*,\Omega^* }\right),
    \end{aligned}
    \end{equation} 
    \item Let $[U^*, \Omega^*],[U_1^*, \Omega_1^*],[U_2^*, \Omega_2^*] \in S^{+}$ and let $p \in \mathcal M_{\mathrm{Krus}}$. Then 
    \begin{equation}\label{order6}
    \begin{aligned}
        [U_1^*, \Omega_1^*]\preceq [U_2^*, \Omega_2^*]\quad&\iff\quad
\Omega^*_1=\Omega^*_2\quad \mathrm{and}\quad U^*_1= U^*_2,\\
        p \preceq [U^*, \Omega^*]\quad&\iff \quad  p \in \widehat {J^-}\left( \gamma^{U^*,\Omega^* }\right) .
        \end{aligned}
    \end{equation}
\end{itemize}
\end{theorem}

No comparison other than those listed in \eqref{order1}--\eqref{order6} and those obtained from them by reflexivity and transitivity holds in $(\overline{\mathcal M}_{\mathrm{Krus}},\preceq)$. In particular, the relations \eqref{order1}--\eqref{order6} define a partial order: reflexivity holds by \eqref{order1} on $\krus$, by the second parts of \eqref{order2} and \eqref{order3} at $i^+$ and $i^{+'}$, and by the first identities in \eqref{order4}--\eqref{order6} on $\mathcal I^+$, $\mathcal I^{+'}$ and $S^+$; antisymmetry and transitivity will also follow from Theorem \ref{main theorem}, the map $T$ constructed in its proof being a bijection onto $(\mathrm Y,\subseteq)$ which turns $\preceq$ into the inclusion.

The rest of the section is devoted to the demonstration of Theorem \ref{main theorem}. We conclude this subsection by discussing the strategy of the proof. 

Let $D \subseteq \krus$ be a directed set. Since $U$ and $V$ are monotone, their images are directed subsets of the extended real line; hence
\begin{equation}\label{eq:existence_limit_UV}
    U^*:=\sup_{p\in D} U(p)\in \mathbb{R}\cup\{+\infty\},\qquad  V^*:=\sup_{p\in D} V(p)\in \mathbb{R}\cup\{+\infty\}.
\end{equation}
We define the following partially ordered set
\begin{equation}\label{eq:def:Y}
    \mathrm Y:=\big\{\widehat{D}: D\subseteq \mathcal M_{\mathrm{Krus}}\text{ directed}\big\}
\end{equation}
endowed with the inclusion.
The proof of Theorem \ref{main theorem} is now divided into two parts.

First, in Subsection \ref{subsection: characterization of directed sets}, we show that the elements of $\mathrm{Y}$ are of the form $\widehat{J^-}(\gamma)$, where $\gamma$ is a suitable radial null geodesic, see Proposition \ref{prop:structure_Dhat_general}. This is achieved by exploiting a case-by-case analysis in dependence of $U^*, V^* \in \mathbb{R}\cup\{+\infty\}$. 

After that, in Subsection \ref{subsection: The structure of the completion}, we prove that $(\mathrm Y,\subseteq)$ is directed complete and it is isomorphic to the directed completion of $\krus$.

Let us explain why proving that $(\mathrm Y,\subseteq)$ is directed complete is enough to identify it with the directed completion of $\krus$. In the construction of the directed completion given in \cite[Theorem 2.5]{gigli2025hyperbolicbanachspaces}, one embeds $\krus$ into the partial order of all lower and directed-sup-closed subsets by $\iota(p):=J^-(p)$ and then adds directed suprema by transfinite recursion; the outcome of the first step of the recursion is precisely
\[
\iota(\krus)^{\uparrow}=\big\{\widehat D : D\subseteq\krus\text{ directed}\big\}=\mathrm Y,
\]
and the supremum of a family in the ambient order is the operation $\{A_i\}\mapsto\widehat{\bigcup_i A_i}$. Hence, if the second step of the recursion produces nothing new, that is if
\begin{equation}\label{eq:first_step_suffices}
    \widehat{\bigcup_{i\in I}\widehat{D_i}}\in\mathrm Y
    \qquad\text{for every directed family }\{\widehat{D_i}\}_{i\in I}\subseteq\mathrm Y,
\end{equation}
then the recursion stabilizes at the first step and $\mathrm Y$, together with $\iota$, is the directed completion of $(\krus,\leq)$. This is exactly the strategy followed for the Minkowski spacetime in \cite[Example 14]{gigli2025hyperbolicbanachspaces}. We shall therefore prove Lemma \ref{lemma:Y_directed_complete} in the sharper form \eqref{eq:first_step_suffices}.

Let us stress that the general criterion of \cite[Proposition 2.10]{gigli2025hyperbolicbanachspaces} is not applicable here: it requires $\widehat D$ to be directed whenever $D$ is, and this fails in $\krus$. Indeed, by Lemma \ref{lemma:timelike_infty} the set $\{U\leq 0\}$ belongs to $\mathrm Y$, while two distinct points of the horizon $\{U=0\}$ with different angular coordinates admit no common upper bound within $\{U\leq0\}$.

\subsection{The characterization of directed sets} \label{subsection: characterization of directed sets}

In this subsection we carry out the first part of our plan. Namely, we characterize the directed sets in terms of radial null curves. In the following proposition we collect all the possibilities.

\begin{proposition}\label{prop:structure_Dhat_general}
    Let $D\subseteq\krus$ be a directed set. Then, there exists $\Omega^*\in \mathbb{S}^2$ such that
    \begin{equation}
        \widehat{D}=\widehat{J^-}(\gamma^{U^*,V^*,\Omega^*}),
    \end{equation}
    where $U^*,V^*$ are as in \eqref{eq:existence_limit_UV} and
    \begin{equation*}
       \widehat{J^-}(\gamma^{U^*,V^*,\Omega^*}):=\begin{cases}
            \widehat{J^-}\big(\{\gamma^{U^*,\Omega^*}(s):s\in (1/U^*,V^*)\}\big),&\quad\mathrm{if }\,U^*\in(-\infty,0), V^*\in(-\infty,+\infty],\\
           \widehat{J^-}\big(\{\gamma^{U^*,\Omega^*}(s): s\in (-\infty,V^*)\}\big),&\quad\mathrm{if }\,U^*\in[0,+\infty),V^*\in(-\infty,+\infty],\\
           \widehat{J^-}\big(\{\gamma^{V^*,\Omega^*}(s): s\in (1/V^*,+\infty)\}\big),&\quad\mathrm{if }\,U^*=+\infty,V^*\in(-\infty,0],
        \end{cases}
    \end{equation*}
     with the convention $1/V^*=-\infty$ if $V^*=0$.
\end{proposition}

Proposition \ref{prop:structure_Dhat_general} follows by a case-by-case analysis and we split the proof into the following cases:
\begin{itemize}
    \item in Subsubsection \ref{subsection:limit_outside}, we consider directed sets admitting a supremum in $\krus$, 
    \item in Subsubsection \ref{subsection:limit_infinity}, we consider directed sets diverging to future null infinity $\mathcal{I}^{+} \cup \mathcal{I}^{+'}$,
    \item in Subsubsection \ref{subsection:limit_timelike}, we consider the directed sets diverging to future timelike infinity $\{i^{+}\} \cup \{i^{+'}\}$,
    \item finally, in Subsubsection \ref{subsection:limit_inside}, we consider directed sets that end in the black hole singularity $S^+$.
\end{itemize}

\subsubsection{Points in the spacetime }\label{subsection:limit_outside}

We start our case-by-case analysis with the situation in which the directed set $D \subseteq \krus$ admits a supremum in $\krus$.

\begin{lemma}\label{lemma:directed_set_tip_interiorM}
    Let $D\subseteq \mathcal M_{\mathrm{Krus}}$ be a directed set, let $U^*,V^*\in\R$ be defined as in \eqref{eq:existence_limit_UV} and assume one of the following \begin{enumerate}
        \item[i)] $U^*\leq 0<V^*$,
        \item[ii)] $V^*\leq 0<U^*$,
        \item[iii)] $U^*>0$, $V^*>0$ and $U^*V^*<1$,
        \item[iv)] $U^*\leq 0$, $V^*\leq 0$ and $U^*V^*<1$.
    \end{enumerate}
    Then there exists $\Omega^*\in \mathbb{S}^2$ such that $\lim_{p\in D}\Omega(p)=\Omega^*$. Let $p^* \in \krus$ be the point identified with the coordinates $(U^*,V^*,\Omega^*)$. Then, we have that $p^*=\sup D$ and that
\begin{equation}\label{eq:point_is_sup:curve}
        \widehat D=J^-(p^*)=
            \widehat{J^-}\left (\{\gamma^{U^*,\Omega^*}(s):s \leq V^*)\}\right) \, .
    \end{equation}
\end{lemma}

\begin{proof}
We prove $i)$, case $ii)$ follows from the symmetry \eqref{eq:symmetric_map_UV}.
Let $p_1\leq p_2$ be points in $A:=\{U\leq 0\,,\,V>0\}$
and let $\gamma :[0,1]\to A$ be a smooth future directed causal curve such that $\gamma_0=p_1$ and $\gamma_1=p_2$.

Along $\gamma_s$ we have $\dot U_s,\dot V_s\geq0$ and $r_s\geq2M$ and, by causality, we have $r_s^2|\dot\Omega_s|^2
    \leq
    \frac{32M^3}{r_s}e^{-\frac{r_s}{2M}}\dot U_s\dot V_s.$
Since $r_s\geq2M$, it follows that
\begin{equation} \label{eq: est omeg}
    |\dot\Omega_s|
    \leq C_M\sqrt{\dot U_s\dot V_s}
    \leq \frac{C_M}{2}(\dot U_s+\dot V_s),
\end{equation}
for some constant $C_M>0$ depending only on $M$. Hence
\[
    \operatorname{dist}_{\mathbb{S}^2}\bigl(\Omega(p_1),\Omega(p_2)\bigr)
    \leq  \int_0^1 |\dot\Omega_s|\,\d s\leq
    \frac{C_M}{2}
    \bigl(U(p_2)-U(p_1)+V(p_2)-V(p_1)\bigr).
\]
Therefore $\Omega(p)$ is a Cauchy net in $\mathbb{S}^2$, and so the conclusion follows.


We now treat the cases $iii)$ and $iv)$, in which the relevant estimate is the one available inside the horizons. We first record that, if $D'\subseteq D$ is cofinal, then $U^*,V^*$ are unchanged and $\widehat{D'}=\widehat D$: indeed $\widehat{D'}\subseteq\widehat D$, while every $p\in D$ lies below some $q\in D'$ and $\widehat{D'}$ is a lower set, so that $D\subseteq\widehat{D'}$ and hence $\widehat D\subseteq\widehat{D'}$. We may therefore always replace $D$ by a cofinal subset.

\emph{Step 1: the case $r<2M$.}
Assume that $D$ is contained either in the black-hole region
$\{U>0,V>0\}$ or in the white-hole region $\{U<0,V<0\}$. Let
$p_1\le p_2$ be points of $D$ and let
$\gamma:[0,1]\to\mathcal M_{\rm Krus}$ be a future-directed causal
curve joining them, with $r_s:=r(\gamma_s)\in(0,2M)$. Along $\gamma$,
$r$ is nonincreasing in the black-hole region and nondecreasing in the
white-hole region. Taking logarithms in \eqref{relation uv} and differentiating gives
\[
\frac{\dot U_s}{U_s}+\frac{\dot V_s}{V_s}
=\frac{r_s}{2M(r_s-2M)}\dot r_s.
\]
Combining this identity with causality and
$ab\le (a+b)^2/4$ yields
\[
|\dot\Omega_s|^2
\le \frac{\dot r_s^2}{r_s(2M-r_s)}.
\]
Consequently,
\[
\operatorname{dist}_{S^2}\bigl(\Omega(p_1),\Omega(p_2)\bigr)
\le
\left|
\int_{r(p_1)}^{r(p_2)}
\frac{dr}{\sqrt{r(2M-r)}}
\right|.
\]
In either region, $r$ is monotone on $D$ and converges to the value
$r^*\in[0,2M)$ determined by $U^*V^*$ through \eqref{relation uv}. Since the
integrand is integrable on $(0,2M)$, the right-hand side tends to zero
along the net. Hence $\Omega$ is a Cauchy net in $\mathbb{S}^2$, and
$\Omega^*:=\lim_{p\in D}\Omega(p)$ exists.

\emph{Step 2: reduction to Step 1.} In case $iii)$, we have $U^*, V^*>0$, so by directedness there is $p_0\in D$ with $U(p_0)>0$ and $V(p_0)>0$; replacing $D$ by the cofinal subset $\{p\in D: p\geq p_0\}$ we may assume $U(p)>0$ and $V(p)>0$, hence $r(p)<2M$, for every $p\in D$, and Step 1 applies. The same argument applies in case $iv)$ whenever $U^*<0$ and $V^*<0$, the roles of the black hole region being played by the white hole region $\{U<0,V<0\}$.

\emph{Step 3: the horizons.} There remains the case $U^*V^*=0$ of $iv)$, in which points of $D$ may lie on the horizons $\{r=2M\}$ and the estimate of Step 1 degenerates. We argue by approximation from below. Let $p_1\leq p_2$ be points of $D$ and, for $n\in\N$ large enough, let $p_i^{(n)}$ be the point of coordinates $\bigl(U(p_i)-\tfrac1n,V(p_i)-\tfrac1n,\Omega(p_i)\bigr)$. Such a shift maps the region $\{U\leq0,V\leq0\}$ into $\{U<0,V<0\}\subseteq\{r<2M\}$ and is order preserving there: it leaves $\dot U,\dot V,\dot\Omega$ unchanged along a causal curve while increasing $UV$, hence decreasing $r$ by \eqref{relation uv}, and the causality condition $\tfrac{32M^3}{r}e^{-r/2M}\dot U\dot V\geq r^2|\dot\Omega|^2$ is preserved because its left hand side is decreasing and its right hand side is increasing in $r$. Therefore $p_1^{(n)}\leq p_2^{(n)}$, and Step 1 applies to them:
\[
\begin{aligned}
\operatorname{dist}_{S^2}\bigl(\Omega(p_1),\Omega(p_2)\bigr)=\operatorname{dist}_{S^2}
\bigl(\Omega(p_1^{(n)}),\Omega(p_2^{(n)})\bigr)\le
\left|
\int_{r(p_1^{(n)})}^{r(p_2^{(n)})}
\frac{\d r}{\sqrt{r(2M-r)}}
\right|.
\end{aligned}
\]
the first equality holding because the shift does not change the angular coordinate. Letting $n\to+\infty$, and using that $r(p_i^{(n)})\to r(p_i)$ and that the integrand is integrable at $r=2M$, we recover the estimate of Step 1 for $p_1,p_2$ themselves. Hence $\Omega$ is a Cauchy net on $D$ and $\Omega^*$ exists.



We conclude the proof by showing that $p^*=\sup D$ and that $\widehat D=J^-(p^*)$. Observe first that the coordinate functions $U$ and $V$ being monotone along $D$ and $\Omega$ being convergent by the first part of the proof, the net $D$ converges to $p^*$ in $\krus$.

\emph{The identity $p^*=\sup D$.}
Fix $p\in D$ and consider the cofinal tail $D_p:=\{q\in D:p\le q\}.$
The net $D_p$ converges to $p^*$. Since
$D_p\subseteq J^+(p)$ and $J^+(p)$ is closed, we obtain
$p\le p^*$. Thus $p^*$ is an upper bound of $D$.
Let $\widetilde p$ be any upper bound of $D$. Then
$D\subseteq J^-(\widetilde p)$. Since $J^-(\widetilde p)$ is closed
and $D$ converges to $p^*$,  we obtain $p^*\le\widetilde p$. Hence
$p^*=\sup D$.

\emph{The identity $\widehat D=J^-(p^*)$.} For the inclusion $\subseteq$, note that $J^-(p^*)$ is a lower set containing $D$ and that it is directed-sup-closed: if $E\subseteq J^-(p^*)$ is directed and admits a supremum, then $p^*$ is an upper bound of $E$ and therefore $\sup E\leq p^*$. By minimality of $\widehat D$ we conclude $\widehat D\subseteq J^-(p^*)$. For the inclusion $\supseteq$, observe that $D$ is directed, is contained in $\widehat D$ and admits the supremum $p^*$; since $\widehat D$ is directed-sup-closed, $p^*\in\widehat D$, and since $\widehat D$ is a lower set, $J^-(p^*)\subseteq\widehat D$.
\end{proof}

\subsubsection{Future timelike infinity}\label{subsection:limit_timelike}

In this subsection, we consider the directed sets that approach future timelike infinity. Recall that, since the map $\mathcal{T}$ which we defined in \eqref{eq:symmetric_map_UV}
acts as a time-orientation-preserving isometry on $\krus$, we can restrict our attention to $i^+$.

\begin{lemma}\label{lemma:timelike_infty} Let $D\subseteq \krus$ be a directed set such that $U^*=0, V^*= +\infty$. Then we have that $\widehat D = \{ U \leq 0\}$ and
\begin{equation} \label{eq: many angles}\widehat{D}=\widehat{J^-}\left (\gamma^{U^*,\Omega}\right),\quad \mathrm{for\, all } \,\, \Omega\in \mathbb{S}^2. 
\end{equation} 
\end{lemma}

\begin{proof}
It is enough to show that $\{ U \leq 0\} \subseteq \widehat D$, since the other inclusion is trivial. Let $p_1$ be a point with $U(p_1)<0$. We claim that there exists some $p_2 \in \widehat D$ such that $p_1\leq p_2$. Once this claim proved,  the conclusion follows. Without loss of generality we assume that $r(p_1) >2M$ and, by the identity in \eqref{relation uv}, there exists $\lambda_*>0$ such that
\[
    \frac{32M^3}{r(p_1)}e^{-\frac{r(p_1)}{2M}}
    \bigl(-\lambda_*^2U(p_1)V(p_1)\bigr)
    >
    r^2(p_1)\pi^2 .
\]
Since $U^*=0$ and $V^*=+ \infty$, there exists
$p_2\in D$ such that $U(p_2)>U(p_1)e^{-\lambda_*}$ and $V(p_2)>V(p_1)e^{\lambda_*}$. Let $q \in \krus$ be the point with coordinates $(U(p_1)e^{-\lambda_*},V(p_1)e^{\lambda_*},\Omega(p_2))$, and note that, by \eqref{relation uv}, we have that $r(q)=r(p_1)$.
Let $\sigma$ be a minimizing geodesic in $\mathbb{S}^2$ from
$\Omega(p_1)$ to $\Omega(p_2)$ and consider 
\[
    \gamma_s:
    =
    \left (U(p_1)e^{-s\lambda_*},
          V(p_1)e^{s\lambda_*},
          \sigma_s\right ).
\]
Notice that $\gamma :[0,1] \rightarrow \krus$ is a causal curve by the choice of $\lambda_*$, connects $p_1$ to $q$ and, by \eqref{relation uv}, $r(\gamma_s)$ is constantly equal to $r(p_1)$. 

Let $\tilde \gamma:[0,1]\to \krus$ be the curve defined, in coordinates, by
\[
    \tilde \gamma_s
    :=
    \bigl((1-s)U(q)+sU(p_2),
          (1-s)V(q)+sV(p_2),
          \Omega(p_2)\bigr).
\]
It is clear that $\tilde \gamma$ joins $q$ to $p_2$ and it is a future-directed causal curve. Concatenating $\gamma$ with $\tilde \gamma$, we obtain $p_1\leq p_2$. Hence, the claim follows and the identity \eqref{eq: many angles} is an easy consequence of the construction.
\end{proof}



\subsubsection{Future null infinity}\label{subsection:limit_infinity}

In this subsection, we consider the directed sets that approach \say{future null infinity}. Recall that, since the map $\mathcal{T}$ which we defined in \eqref{eq:symmetric_map_UV}
acts as a time-orientation-preserving isometry on $\krus$, we can restrict our attention to $\mathcal{I}^+$.

\begin{lemma}\label{lemma:future_null_angle}
Assume that $D\subseteq \krus$ is a directed set such that $U^*\in(-\infty,0), V^*=+\infty$. Then there exists $\Omega^*\in \mathbb{S}^2$ such that $\lim_{p\in D} \Omega(p)=\Omega^*$. Moreover, we have that
\[
    \widehat D =
    J^-\left (\gamma^{U^*,\Omega^*} \right )
    =
    \widehat{J^-}\left (\gamma^{U^*,\Omega^*}\right ).
\]
    \end{lemma}
    
\begin{proof} Possibly replacing $D$ with the smallest lower set containing it, it is not restrictive to assume that $D$ is also a lower set.  Now fix $\eps>0$ and notice  that since the assumptions combined with \eqref{relation uv} imply that \(\lim_{p \in D} r(p) = +\infty\),  there is $p\in D$ such that
\begin{equation}
\label{eq:pgrande}
r(p')>\max\{8M,\eps^{-1}\},\quad V(p')>0,\quad \log\left(\tfrac{U(p')}{U^*}\right)<\eps,\qquad\forall\, p'\in D,\ p\leq p'.
\end{equation}
Fix such $p$, let $p_1,p_2\in D$ be such that $p\leq p_1 \leq p_2$ and  then let $\gamma:[0,1]\to\krus$ be a future directed causal curve joining \(p_{1}\) to \(p_{2}\). Put for brevity $U_s:=U(\gamma_s)$ and define similarly $V_s,r_s,\Omega_s$ and notice that since $D$ is a lower set we have $\gamma_s\in D$ for every $s\in[0,1]$, so that \eqref{eq:pgrande} holds for $p':=\gamma_s$.

Since $\gamma$ is a causal curve we have $\frac{32M^3}{r^3_s}e^{-\frac{r_s}{2M}}\,\dot{U}_s\dot{V}_s-|\dot\Omega_s|^2\ge 0.
$ 
Combining this with \eqref{relation uv} and Young's inequality, we get
\begin{equation}\label{Omega es}
    |\dot\Omega_s|\le
4M\,\frac{\sqrt{r_s -2M}}{r_s^{ 3/2}}
\frac{\sqrt{\dot{U}_s\dot{V}_s}}{\sqrt{-U_s V_s }}\leq\frac{4M}{r_s}\sqrt{\frac{\dot U_s}{-U_s}}\sqrt{\frac{\dot V_s}{V_s}}\leq 2M\left(
-\frac{\dot{U}_s}{\ U_s }+\frac{\dot{V}_s }{r_s^2V_s}
\right).
\end{equation}
Taking logs in \eqref{relation uv} and differentiating along $\gamma_s$  we obtain 
\begin{equation}\label{eq:computation_r_dot}
    \frac{\dot{U}_s}{U_s}+    \frac{\dot{V}_s}{V_s}
=
\frac{r_s }{2M(r_s -2M)}\,\dot r_s.
\end{equation}
Recalling that $-U_s,\dot U_s,V_s,\dot V_s,r_s-2M>0$ we deduce that 
\begin{equation}
\label{eq:rprimo}    
\dot r_s 
\ge
\frac{2M(r_s -2M )}{r_s }\frac{\dot U_s}{U_s }\geq 2M\frac{\dot U_s }{U_s }.
\end{equation}
Combining \eqref{Omega es} and \eqref{eq:computation_r_dot}, we get $
 |\dot\Omega_s|\leq 2M\big(-\frac{\dot U_s}{U_s}(1+\frac{1}{r_s^2})+\frac{\dot r_s}{2Mr_s(r_s-2M)}\big)
$
so that, taking into account \eqref{eq:pgrande}, we have
\[
 |\dot\Omega_s|\leq 2M\left(-\tfrac{\dot U_s}{U_s}\right)(1+4M^{-2})+\tfrac{\dot r_s}{r_s(r_s-2M)},
\]
where we used that $r_s\geq 8M$ gives $1+r_s^{-2}\leq 1+4M^{-2}$. We do not estimate the last term pointwise: its sign is not fixed, because along a future directed causal curve in the exterior region $\dot r_s$ may be negative, and multiplying by the two different positive denominators $r_s(r_s-2M)$ and $r_s^2$ would then reverse the desired comparison. We integrate it exactly instead. Set
\[
    F(r):=\frac{1}{2M}\log\Bigl(1-\frac{2M}{r}\Bigr),\qquad r>2M,
\]
so that $F'(r)=1/[r(r-2M)]$ and $F\leq0$. Then
\[
    \int_0^1\frac{\dot r_s}{r_s(r_s-2M)}\,\d s
    =
    F\bigl(r(p_2)\bigr)-F\bigl(r(p_1)\bigr)
    \leq
    -F\bigl(r(p_1)\bigr)
    =
    \frac{1}{2M}\log\frac{r(p_1)}{r(p_1)-2M}
    \leq
    \frac{2}{r(p_1)},
\]
the last inequality following from $\log(1+x)\leq x$ and from $r(p_1)\geq8M$, which gives $r(p_1)-2M\geq\tfrac34 r(p_1)$ and therefore $\tfrac1{2M}\log\bigl(1+\tfrac{2M}{r(p_1)-2M}\bigr)\leq\tfrac{4}{3r(p_1)}$. Since moreover $\int_0^1(-\dot U_s/U_s)\,\d s=\log\bigl(U(p_1)/U(p_2)\bigr)\leq\log\bigl(U(p_1)/U^*\bigr)$, because $U(p_2)\leq U^*<0$,
it follows that
\begin{equation}\label{estimate of angle}
    \operatorname{dist}_{\mathbb{S}^2}\bigl(\Omega(p_1),\Omega(p_2)\bigr)\leq 2M(1+4M^{-2})\log\left(\tfrac{U(p_1)}{U^*}\right)+\tfrac2{r(p_1)}\stackrel{\eqref{eq:pgrande}}\leq C_M\eps.
\end{equation}
Therefore $\Omega(p)$ is a Cauchy net in $\mathbb{S}^2$, and so the first claim follows.

Now we move to the second part. 
We first prove that $J^-(\gamma^{U^*,\Omega^*})\subseteq \widehat D$.
Let $(p_i)_{i\in \N}\subset D$ be a sequence such that $p_i\leq p_{i+1}$ for every $i$ and
\begin{equation} \label{eq: conv}
    U(p_i):= U_i\to U^*,\qquad  V(p_i):=V_i\to+\infty,\qquad  \Omega (p_i):=\Omega_i\to \Omega^* .
\end{equation}
Since $\gamma^{U_i,\Omega_i}(V_i)=p_i$ and $D$ is a lower set, we have $\gamma^{U_i,\Omega_i}((1/U_i,V_i])\subseteq D$. Therefore,
\[
    J^-\!\left(\gamma^{U_i,\Omega_i}\bigl((1/U_i,V_i]\bigr)\right)
    \subseteq \widehat D .
\]
We now let $i\to+\infty$. Since $\widehat D$ is directed-sup-closed and every point of $\gamma^{U^*,\Omega^*}$ is the supremum of an increasing sequence of points lying strictly in its chronological past, it is enough to show that every $q$ with $U(q)<U^*$ and $q\in J^-(\gamma^{U^*,\Omega^*})$ belongs to $\widehat D$. For such a $q$ we have $U(q)<U_i$ for $i$ large, while $\Omega_i\to\Omega^*$ and $V_i\to+\infty$; the angular gap between $q$ and $p_i$ can then be covered by a rotation at constant $r$ followed by a radial null segment, so that $q\leq p_i$ and hence $q\in\widehat D$.

Conversely, fix $p\in D$. Since $D$ is directed, there exists a sequence $(p_i)_{i\in \N}\subset D$ 
such that $p\leq p_i\leq p_{i+1}$ for every $i$ and \eqref{eq: conv} holds. Then,
\begin{equation}\label{eq:increasing_geodesics}
    p\in J^-(p_i)=J^-\!\big(\gamma^{U_i,\Omega_i}\bigl((1/U_i,V_i]\bigr)\big)\subset J^-\left (\gamma^{U^*,\Omega^*}\right )
\end{equation}
Hence, $D\subseteq J^-(\gamma^{U^*,\Omega^*}) =\widehat{J^-}(\gamma^{U^*,\Omega^*})$, where the last identity follows since $J^-(\gamma^{U^*,\Omega^*})$ is directed-sup-closed since $V^*=+ \infty$.
\end{proof}


In the next corollary, we provide the inclusion relation among directed sets reaching future null infinity.

\begin{corollary}\label{cor:future_null_order}
Assume that $D_i \subseteq \krus$, $i=1,2$, are two directed sets such that $\sup_{p\in D_i} U(p)=U_i^*\in (-\infty,0)$ and $\sup_{p\in D_i} V(p)=V_i^*=+\infty$. By Lemma \ref{lemma:future_null_angle}, there exists $\Omega_i^*\in \mathbb{S}^2$ such that $\lim_{p\in D_i} \Omega(p)=\Omega_i^*$. Then
\[
\widehat D_1\subseteq \widehat D_2
\qquad\Longleftrightarrow\qquad
\Omega^*_1=\Omega^*_2\quad \mathrm{and}\quad U^*_1\le U^*_2.
\]
In particular,
\[
\widehat D_1=\widehat D_2
\qquad\Longleftrightarrow\qquad
\Omega^*_1=\Omega^*_2\quad \mathrm{and}\quad  U^*_1=U^*_2.
\]
\end{corollary}

\begin{proof}
If $\Omega_1^*=\Omega_2^*=:\Omega^*$ and $U_1^*\leq U_2^*$, then by Lemma \ref{lemma:future_null_angle} it is easy to see that 
\[ 
\widehat D_1 = J^-(\gamma^{U_1^*,\Omega^*}) \subseteq J^-(\gamma^{U_2^*,\Omega^*}) = \widehat D_2 \, .
\]
In particular, if $U_1^*=U_2^*$, then $\widehat D_1=\widehat D_2$.

Conversely, assume that $\widehat D_1\subseteq \widehat D_2$. Then, since $U$ satisfies the {\sf Mcp}, $U_1^*\leq U_2^*$. Now, we claim that $\Omega_1^*=\Omega_2^*$. Let $(p_n)_{n \in \N} \subseteq D_1$ such that
\[
    U(p_n) <U_1^*,\quad U(p_n)\to U_1^*,\quad V(p_n)\to+\infty .
\]
Since $\widehat D_1\subseteq\widehat D_2$, we may find a point $q_n$ on the support of $\gamma^{U_2^*,\Omega_2^*}$ such that $
    p_n\leq q_n, V(q_n)\to+\infty.$
We now distinguish two cases. If $U_1^*=U_2^*$, then applying \eqref{estimate of angle} to $p_n\leq q_n$, for $n$ large enough, gives
\[
 \operatorname{dist}_{\mathbb{S}^2}\bigl(\Omega(p_n),\Omega(q_n)\bigr)\leq 2M(1+4M^{-2})\log\left(\tfrac{U(p_n)}{U_2^*}\right)+\tfrac2{r(p_n)}\xrightarrow{n\to+\infty} 0,
\]
which implies $\Omega^*_1=\Omega^*_2$. If $U_1^*<U_2^*$, suppose by contradiction that $ \delta:=\operatorname{dist}_{\mathbb{S}^2}(\Omega_1^*,\Omega_2^*)>0.$
Then, for $n$ large enough, changing only the parameter in Young's inequality and arguing as in the derivation of \eqref{estimate of angle}, we obtain
\begin{equation*}
     \mathrm{dist}_{\mathbb{S}^2}\bigl(\Omega(p_n),\Omega(q_n)\bigr)
\le \frac{\delta}{2\log\frac{U^*_1}{U^*_2}} \, \log\left ( \tfrac{U(p_n)}{U^*_2} \right )+\tfrac C{r(p_n)},
\end{equation*}
where $C$ is a constant depending only on $\delta$, $M$, $U_1^*$ and $U_2^*$. Therefore,
\begin{equation*}
        0<\delta=  \operatorname{dist}_{\mathbb{S}^2}(\Omega^*_1,\Omega^*_2)=\lim_{n\to +\infty}  \operatorname{dist}_{\mathbb{S}^2}\bigl(\Omega(p_n),\Omega(q_n)\bigr)\leq \tfrac{\delta}{2},
\end{equation*}
which is a contradiction. This completes the proof.
\end{proof}

\subsubsection{The black hole singularity}\label{subsection:limit_inside}

In the following lemma we discuss directed sets that end in the black hole singularity. Notice that, unlike the situation in Lemma \ref{lemma:future_null_angle}, we have that $U^*\in(0,+\infty), V^* =\tfrac{1}{U^*}$ and also $\widehat{J^-}(\gamma^{U^*,\Omega^*}) \neq {J^-}(\gamma^{U^*,\Omega^*})$.

\begin{lemma}\label{lemma:singularity-angle}
Assume that $D \subseteq \krus$ is a directed set such that $U^*\in(0,+\infty), V^* =\tfrac{1}{U^*}$. Then, there exists $\Omega^*\in \mathbb{S}^2$ such that $\lim_{p\in D} \Omega(p)=\Omega^*$ and $\widehat D=\widehat{J^-}(\gamma^{U^*,\Omega^*}).$
\end{lemma}

\begin{proof}
Without loss of generality, we may assume that $p_1,p_2\in D$ are such that $p_1 \leq p_2$ and \(r(p_{1}), r(p_{2}) \in (0,2M)\).
Let $\gamma:[0,1]\to\krus$ be a future directed causal curve joining \(p_{1}\) to \(p_{2}\) such that \(r_s:=r(\gamma_s)\in(0,2M)\). Put, for brevity, $U_s:=U(\gamma_s)$, and define $V_s$ and $\Omega_s$ similarly. Then 
\begin{equation}\label{causal curve}
    \frac{32M^3}{r_s}e^{-\frac {r_s}{2M}}\dot U_s\dot V_s-r^2_s|\dot\Omega_s|^{2}\ge0.
\end{equation}
Along \(\gamma_s\), we have $ 0<r_s<2M, U_s>0,  V_s>0, \dot U_s\ge 0, \dot V_s\ge 0.$
In particular, for any $s\in [0,1]$, $\dot r_s<0,$ otherwise, $\gamma_s$ would fail to be causal.
Combining \eqref{relation uv}, \eqref{eq:computation_r_dot} and \eqref{causal curve}, we obtain
\begin{equation}\label{less than}
    |\dot\Omega_s|^{2}
\le
\frac{\dot r^{2}_s}{r_s(2M-r_s)}.
\end{equation}
In particular, we have that
\begin{equation}\label{estimate of omega}
     \operatorname{dist}_{\mathbb{S}^2}\bigl(\Omega(p_1),\Omega(p_2)\bigr)
    \leq
    \int_{r(p_2)}^{r(p_1)}
    \frac{\mathrm{d}r}{\sqrt{r(2M-r)}},
\end{equation}
and so $\Omega(p)$ is a Cauchy net in
$\mathbb{S}^2$, hence the first part follows.

We are left to show the final identity. We first prove that $\widehat{J^-}(\gamma^{U^*,\Omega^*})\subseteq\widehat D$. Since $\widehat D$ is directed-sup-closed and every point of $\gamma^{U^*,\Omega^*}$ is the supremum of an increasing sequence of points lying strictly in its chronological past, it suffices to show that every $q$ with $U(q)<U^*$ and $q\in J^-(\gamma^{U^*,\Omega^*})$ belongs to $\widehat D$. Choose an increasing sequence $(p_i)_{i\in\N}\subseteq D$ with $U(p_i)\to U^*$, $V(p_i)\to1/U^*$ and $\Omega(p_i)\to\Omega^*$. Then $U(q)<U(p_i)$ for $i$ large, and the angular gap between $q$ and $p_i$ can be covered by a rotation at constant $r$ followed by a radial null segment, so that $q\leq p_i$ and hence $q\in\widehat D$. Conversely, fix $p\in D$. For $n$ sufficiently large, set $p_n^-:=(U(p)-1/n,\,V(p)-1/n,\,\Omega(p)).$
Then $p_n^-\in I^-(p)$. By \eqref{eq:increasing_geodesics},
\[
p_n^-\in
\widehat{J}^{-}\bigl(\gamma^{U^*,\Omega^*}\bigr)
\qquad\text{for every sufficiently large }n.
\]
The sequence $(p_n^-)$ is increasing and converges to $p$; by
Lemma \ref{lemma:directed_set_tip_interiorM}, $p=\sup_n p_n^-$. Since
$\widehat{J}^{-}(\gamma^{U^*,\Omega^*})$ is directed-sup-closed, we obtain $p\in\widehat{J}^{-}\bigl(\gamma^{U^*,\Omega^*}\bigr).$
Thus
$\widehat D\subseteq\widehat{J}^{-}(\gamma^{U^*,\Omega^*})$, and the
reverse inclusion proved above yields equality.
\end{proof}

In the next corollary, we provide the inclusion relation among directed sets reaching the black hole singularity.

\begin{corollary}\label{corol:singularity-order-structure}
Let $D_i\subseteq \krus,i=1,2$ be directed sets such that $\sup_{p\in D_i}U(p)=U^*_i\in(0,+\infty)$ and $
\sup_{p\in D_i}V(p)=V^*_i=1/{U^*_i}$. By Lemma \ref{lemma:singularity-angle}, there exists $\Omega_i^*\in \mathbb{S}^2$ such that $\lim_{p\in D_i} \Omega(p)=\Omega_i^*$.  Then
\[
\widehat D_1\subseteq \widehat D_2
\qquad\Longleftrightarrow\qquad
\Omega^*_1=\Omega^*_2\quad \mathrm{and}\quad U^*_1= U^*_2 \qquad\Longleftrightarrow\qquad \widehat D_1=\widehat D_2.
\]
    
\end{corollary}

\begin{proof}
First notice that the last equivalence is trivial by Lemma \ref{lemma:singularity-angle}. We are left to prove that if $\widehat D_1\subseteq \widehat D_2$ then $\Omega^*_1=\Omega^*_2$ and $U^*_1= U^*_2$.

Since $U,V$ satisfy the {\sf Mcp}, we have $ U_1^*\leq U_2^*$ and $ V_1^*\leq V_2^*$. But $V^*_i=1/{U^*_i}$ and so $U_1^*=U_2^*$ and $V_1^*=V_2^*$. It remains to prove that $\Omega_1^*=\Omega_2^*$. Put
$U^*:=U_1^*=U_2^*,
    V^*:=1/U^*.$
Using
\[
    \widehat D_1
    =
    \widehat{J^-}\bigl(\gamma^{U^*,\Omega_1^*}\bigr)
    \subseteq
    \widehat{J^-}\bigl(\gamma^{U^*,\Omega_2^*}\bigr)
    =
    \widehat D_2,
\]
we can find $p_n\in\widehat D_1$ and $q_n\in\widehat D_2$ such that $p_n\leq q_n$ and
\[
    \big(U(p_n),V(p_n),\Omega(p_n)\big)\to (U^*,V^*,\Omega_1^*),\qquad
    \big(U(q_n),V(q_n),\Omega(q_n)\big)\to (U^*,V^*,\Omega_2^*),
\]
which implies that $r(p_n)$ and $r(q_n)$ tend to zero. Applying the estimate \eqref{estimate of omega} to $p_n\leq q_n$
and letting $n\to+\infty$, we obtain $\Omega_1^*=\Omega_2^*$, and the proof is complete.
\end{proof}

\subsection{The structure of the completion} \label{subsection: The structure of the completion}

In the following lemma we show that the partially ordered set $(\mathrm{Y},\subseteq)$ is directed complete.

\begin{lemma}\label{lemma:Y_directed_complete}
    The partially ordered set $(\mathrm{Y},\subseteq)$, defined in \eqref{eq:def:Y}, is directed complete. More precisely, for every directed family
$\{\widehat D_i:i\in I\}\subseteq Y$, the set $F:=\widehat{\bigcup_{i\in I}\widehat D_i}$
belongs to $Y$ and is the supremum of the family in $(Y,\subseteq)$.
\end{lemma}
\begin{proof}
    Let $\{ \widehat{D_i} : i\in I \}$ be a directed family in $\mathrm Y$.
We take the family itself as index set, that is we endow $I$ with the preorder
    \[
        i\leq j\quad\Longleftrightarrow\quad \widehat{D_i}\subseteq\widehat{D_j},
    \]
    which is directed by assumption. We shall use twice the following elementary remarks. First, if $J\subseteq I$ is cofinal, then $\{\widehat{D_i}\}_{i\in J}$ and $\{\widehat{D_i}\}_{i\in I}$ have the same upper bounds and the same union, so that we may always pass to a cofinal tail. Second, in order to prove \eqref{eq:first_step_suffices} it suffices, in each case, to exhibit an element $F\in\mathrm Y$ such that
    \begin{equation}\label{eq:sup_criterion}
        \widehat{D_i}\subseteq F\quad\forall i\in I,
        \,\,\text{and}\,\,
        F\subseteq G\ \text{ for every lower and directed-sup-closed }G\supseteq\textstyle\bigcup_{i\in I}\widehat{D_i};
    \end{equation}
    such an $F$ is precisely the smallest lower and directed-sup-closed set
containing $\bigcup_{i\in I}\widehat D_i$, and therefore it is the
supremum of the family in $Y$.
    
    By Proposition \ref{prop:structure_Dhat_general}, for every $i\in I$, there exists $(U^*_i,V^*_i,\Omega^*_i)$ such that $\widehat{D_i}=\widehat{J^-}(\gamma^{U^*_i,V^*_i,\Omega^*_i})$. Notice that $U^*_i=\sup_{p\in\widehat{D_i}}U(p)$: indeed $D_i\subseteq\widehat{D_i}$, while the set $\{U\leq U^*_i\}$ is a lower set, is directed-sup-closed because $U$ has the {\sf Mcp}, and contains $D_i$, so that it contains $\widehat{D_i}$. The same holds for $V$. In particular $i\leq j$ implies $U^*_i\leq U^*_j$ and $V^*_i\leq V^*_j$, so that the monotone limits
    \[
        U^*_\infty:=\lim_{i\in I}U^*_i\in\R\cup\{+\infty\},
        \qquad
        V^*_\infty:=\lim_{i\in I}V^*_i\in\R\cup\{+\infty\}
    \]
    exist. Observe also that $V^*_i=+\infty$ forces $U^*_i\leq0$: if some $p\in D_i$ had $U(p)=u>0$, then every $q\in D_i$ with $q\geq p$ would satisfy $U(q)\geq u$ and hence $V(q)<1/u$, whence $V^*_i\leq1/u$. We split the treatment in the following cases.\\
    \textbf{Case 1:} If $(U^*_i,V^*_i,\Omega^*_i)$ defines a point $p_i\in \mathcal M_{\mathrm{Krus}}$ for every $i\in I$, then Lemma \ref{lemma:directed_set_tip_interiorM} gives $\widehat{D_i}=J^-(p_i)$, and $D=\{p_i: i\in I\}$ is a directed set in $\mathcal M_{\mathrm{Krus}}$: indeed $i\leq j$ means $J^-(p_i)\subseteq J^-(p_j)$, whence $p_i\leq p_j$.
    We claim that $F:=\widehat D$ satisfies \eqref{eq:sup_criterion}. It belongs to $\mathrm Y$ by the very definition \eqref{eq:def:Y}. It contains every $\widehat{D_i}=J^-(p_i)$, because $p_i\in D\subseteq\widehat D$ and $\widehat D$ is a lower set. Finally, if $G$ is lower, directed-sup-closed and contains $\bigcup_i\widehat{D_i}\supseteq D$, then $\widehat D\subseteq G$ by minimality of the operation $A\mapsto\widehat A$. By monotonicity, there exists $U^*_\infty:=\lim_{i\in I} U^*_i\in\R\cup\{+\infty\}$ and $V^*_\infty:=\lim_{i\in I} V^*_i\in\R\cup\{+\infty\}$, and applying Proposition \ref{prop:structure_Dhat_general} to $D$ we obtain
    \[
    \sup_{i\in I}\widehat{D_i}=\widehat{D}=\widehat{J^-}\big(\gamma^{U^*_\infty,V^*_\infty,\Omega^*_\infty}\big)\in \mathrm Y,
    \]
    where $\Omega^*_\infty$ is the angular parameter provided by Proposition \ref{prop:structure_Dhat_general}, namely $\Omega^*_\infty=\lim_{i\in I}\Omega^*_i$ unless $U^*_\infty=+\infty,V^*_\infty=0$ or $U^*_\infty=0,V^*_\infty=+\infty$, in which cases the limit set does not depend on the angular parameter by Lemma \ref{lemma:timelike_infty} and any $\Omega_0\in \mathbb{S}^2$ may be chosen.\\
    \textbf{Case 2:} Assume that there exists $\overline i\in I$ such that $U^*_{\overline i}V^*_{\overline i}=1$. Then by Corollary \ref{corol:singularity-order-structure}, 
    $U_{\overline{i}}^*= U_{i}^*$, $V_{\overline{i}}^*= V_{i}^*$ and $\Omega_{\overline{i}}^*= \Omega_{i}^*$ holds for any $i \geq \overline{i}$, that is $\widehat{D_i}=\widehat{D_{\overline i}}$ for every $i\geq\overline i$. Since $\{i\in I: i\geq\overline i\}$ is cofinal, for an arbitrary $i\in I$ we may pick $j\geq i,\overline i$ and obtain $\widehat{D_i}\subseteq\widehat{D_j}=\widehat{D_{\overline i}}$; hence $\bigcup_{i\in I}\widehat{D_i}=\widehat{D_{\overline i}}$ and $F:=\widehat{D_{\overline i}}\in\mathrm Y$ trivially satisfies \eqref{eq:sup_criterion}. Therefore $\sup_{i\in I}\widehat{D_i}=\widehat{D_{\overline i}}$.\\
    \textbf{Case 3:} Assume that there exists $\overline i\in I$ such that $V^*_{\overline{i}}=+\infty$ and suppose that $U^*_i<0$ for every $i\in I$ (the case $U^*_{\overline{i}}=+\infty$ and $V^*_i<0$ for every $i\in I$ is analogous). Then, by Corollary \ref{cor:future_null_order}, $V^*_i=+\infty$ and $\Omega^*_i=\Omega^*_{\overline i}=:\Omega^*$ for any $i \geq \overline{i}$. Passing to the cofinal tail $\{i\geq\overline i\}$ and setting $U^*_\infty:=\lim_{i\in I}U^*_i\in(-\infty,0]$, consider
    \[
        D:=\bigl\{\gamma^{U^*_i,\Omega^*}(s) \;:\; i\geq\overline i,\ s\in(1/U^*_i,+\infty)\bigr\}.
    \]
    The set $D$ is directed: two of its points have the same angular coordinate and, by \eqref{eq:increasing_geodesics}, the point $\gamma^{U^*_j,\Omega^*}(s)$ with $j$ an upper bound of the two indices and $s$ larger than both parameters lies in their common causal future. Moreover $\sup_{p\in D}U(p)=U^*_\infty$, $\sup_{p\in D}V(p)=+\infty$ and $\Omega\equiv\Omega^*$ on $D$, so that Lemma \ref{lemma:future_null_angle} (if $U^*_\infty<0$) or Lemma \ref{lemma:timelike_infty} (if $U^*_\infty=0$) gives $\widehat D=\widehat{J^-}(\gamma^{U^*_\infty,+\infty,\Omega^*})$. We claim that $F:=\widehat D$ satisfies \eqref{eq:sup_criterion}. On the one hand $D\subseteq\bigcup_{i}\widehat{D_i}$, because $\gamma^{U^*_i,\Omega^*}$ is contained in $\widehat{D_i}$, so that $F\subseteq G$ for every lower and directed-sup-closed $G$ containing $\bigcup_i\widehat{D_i}$. On the other hand $U^*_i\leq U^*_\infty$ and the angular parameters agree, whence $\widehat{D_i}\subseteq F$ for every $i$. Therefore
    \[
    \sup_{i\in I}\widehat{D_i}=\widehat{J^-}\big(\gamma^{U^*_{\infty},+\infty,\Omega^*}\big)\in \mathrm{Y}.
    \]\\
    \textbf{Case 4:} Assume that there exists $\overline i\in I$ such that  $V^*_{\overline{i}}=+\infty$ and $U^*_{\overline i}=0$ (the case $U^*_{\overline{i}}=+\infty$ and $V^*_{\overline{i}}=0$ is analogous). Then, by Lemma \ref{lemma:timelike_infty}, $\widehat{D_{\overline i}}=\{U\leq0\}$. For $i\geq\overline i$ we have $\widehat{D_i}\supseteq\widehat{D_{\overline i}}$, hence $V^*_i=+\infty$, and therefore $U^*_i\leq0$ by the observation made at the beginning of the proof; since also $U^*_i\geq U^*_{\overline i}=0$, we conclude that $V^*_i=+\infty$ and $U^*_{i}=0$ for any $i \geq \overline{i}$, so that $\widehat{D_i}=\widehat{D_{\overline i}}$ by Lemma \ref{lemma:timelike_infty}. As in Case 2, the tail $\{i\geq\overline i\}$ being cofinal we obtain $\bigcup_{i\in I}\widehat{D_i}=\widehat{D_{\overline i}}$, and $F:=\widehat{D_{\overline i}}$ satisfies \eqref{eq:sup_criterion}. Hence, for any $\Omega_0\in \mathbb{S}^2$,
    $\widehat{J^-}(\gamma^{ 0, +\infty, \Omega_0})= \widehat{D}_{\overline{i}}= \sup_{i\in I}\widehat{D}_i$.

    Finally, the four cases are exhaustive: If every triple
$(U_i^*,V_i^*,\Omega_i^*)$ defines a point of
$\mathcal M_{\rm Krus}$, then we are in Case 1. Otherwise, by Proposition \ref{prop:structure_Dhat_general}, some index $\overline i$ satisfies $U^*_{\overline i}V^*_{\overline i}=1$, or $V^*_{\overline i}=+\infty$, or $U^*_{\overline i}=+\infty$, which are the situations covered by Cases 2, 3 and 4.
    \end{proof}

We conclude this section with the proof of Theorem \ref{main theorem}.

\begin{proof}[Proof of Theorem \ref{main theorem}]
    It suffices to show that  $(\mathrm Y,\subseteq)$
    is isomorphic to $(\overline{\mathcal{M}}_{\mathrm{Krus}}, \preceq)$.
    To do so, using the notation introduced in Proposition \ref{prop:structure_Dhat_general}, we define the map
    \[
    T\colon\,\,\, \overline{\mathcal{M}}_{\mathrm{Krus}} \to \mathrm Y
    \]
    by
    \begin{equation}
        T(p)=\begin{cases}
            \widehat{J^-}(\gamma^{U(p),V(p),\Omega(p)}),\qquad&\text{if }p\in\krus,\\
            \widehat{J^-}(\gamma^{0,+\infty,\Omega_0}),\qquad&\text{if }p=i^+,\\
            \widehat{J^-}(\gamma^{+\infty,0,\Omega_0}),\qquad&\text{if }p={i^+}',\\
            \widehat{J^-}(\gamma^{U^*,+\infty,\Omega^*}),\qquad&\text{if }p=[U^*,\Omega^*]\in{\mathcal I^+},\\
            \widehat{J^-}(\gamma^{+\infty,V^*,\Omega^*}),\qquad&\text{if }p=[V^*,\Omega^*]\in{\mathcal I^+}',\\
            \widehat{J^-} (\gamma^{U^*,1/U^*,\Omega^*} ),\qquad&\text{if }p\in S^+.
        \end{cases}
    \end{equation}

    We check separately that $T$ is well-posed, surjective, injective, order preserving and order reflecting.

    \emph{Well-posedness.} The only definitions involving a choice are those of $T(i^+)$ and $T(i^{+'})$. By Lemma \ref{lemma:timelike_infty} we have $\widehat{J^-}(\gamma^{0,+\infty,\Omega})=\{U\leq0\}$ for every $\Omega\in \mathbb{S}^2$, so that $T(i^+)$ does not depend on the arbitrary $\Omega_0$; the same argument applies to $T(i^{+'})$.

    \emph{Surjectivity.} By Proposition \ref{prop:structure_Dhat_general} every element of $\mathrm Y$ is of the form $\widehat{J^-}(\gamma^{U^*,V^*,\Omega^*})$ with $U^*,V^*\in\R\cup\{+\infty\}$, and the admissible pairs $(U^*,V^*)$ are exactly: those with $U^*V^*<1$ and $U^*,V^*$ finite, which are the coordinates of a point of $\krus$; those with $U^*>0$ and $V^*=1/U^*$, corresponding to $S^+$; those with $V^*=+\infty$ and $U^*<0$, corresponding to $\mathcal I^+$, and symmetrically for $\mathcal I^{+'}$; and those with $V^*=+\infty$, $U^*=0$, corresponding to $i^+$, and symmetrically for $i^{+'}$. Each of them is in the image of $T$.

    \emph{Order preservation and order reflection.} Being $\preceq$ defined by the explicit list \eqref{order1}--\eqref{order6}, the two implications
    \[
        p\preceq q\ \Longrightarrow\ T(p)\subseteq T(q),
        \qquad\qquad
        T(p)\subseteq T(q)\ \Longrightarrow\ p\preceq q
    \]
    are distinct statements, and neither of them follows from bijectivity together with the other: a monotone bijection between partial orders need not be an isomorphism, as one sees already by comparing a two-element set carrying no relation with the same set totally ordered. In our situation, however, most of the verification is automatic, because \eqref{order1} and the first equivalences in \eqref{order4}, \eqref{order5} and \eqref{order6} are stated as equivalences, and so are the descriptions of the inclusions between the sets $\widehat{J^-}(\gamma^{U^*,V^*,\Omega^*})$ given by Lemma \ref{lemma:directed_set_tip_interiorM} within $\krus$, by Corollary \ref{cor:future_null_order} on $\mathcal I^+$ and $\mathcal I^{+'}$, by Corollary \ref{corol:singularity-order-structure} on $S^+$ and by Lemma \ref{lemma:timelike_infty} at $i^+$ and $i^{+'}$. Both implications therefore hold whenever $p$ and $q$ lie in the same piece of $\overline{\mathcal M}_{\mathrm{Krus}}$, and likewise when $p\in\krus$, by the second identities in \eqref{order2}--\eqref{order6}.

    What is left, and is exactly the content of the assertion that no comparison other than those listed occurs, is the case in which $p$ and $q$ belong to two \emph{different} boundary pieces: there \eqref{order1}--\eqref{order6} declare $p$ and $q$ to be incomparable, and order reflection amounts to checking that the corresponding inclusion $T(p)\subseteq T(q)$ indeed fails.

    \emph{Injectivity.} If $T(p)=T(q)$ then, by order reflection, $p\preceq q$ and $q\preceq p$, whence $p=q$.

Two representative cases are the following. Let $p=[U^*,\Omega^*]\in\mathcal I^+$ and $q=[V^*,\Omega_1^*]\in\mathcal I^{+'}$: the set $T(p)$ contains points with arbitrarily large $V$ and $T(q)$ does not, while $T(q)$ contains points with arbitrarily large $U$ and $T(p)$ does not, so that neither inclusion holds and $p,q$ are incomparable, as declared. Let now $p\in S^+$ and $q=i^+$: every point of $T(q)=\{U\leq0\}$ with $V$ large fails to lie in $T(p)$, on which $V<1/U^*$, while $T(p)$ contains points with $U>0$; again neither inclusion holds. The remaining mixed cases are analogous.
\end{proof}

\section*{Data availability statement}

There is no data associated with this manuscript.

\section*{Conflict of interest statement}

The authors confirm that no conflict of interest exists for this work.

\section*{Acknowledgments}
The work of NG, MP, ZX and MZ was supported by the Italian Ministry of University and Research (MUR) under the \say{Fondo Italiano per la Scienza} (FIS 3) program, PI: prof.\ Nicola Gigli, Project Title: Modern challEnges in Geometric Analysis --- MEGA, CUP: G53C25000920001.

This research was funded in part by the Austrian Science Fund (FWF) [Grants DOI \href{https://doi.org/10.55776/EFP6}{10.55776/EFP6} and \href{https://doi.org/10.55776/J4913}{10.55776/J4913}]. For open access purposes, the authors have applied a CC BY public copyright license to any author accepted manuscript version arising from this submission.

The authors would like to thank Maxime Van de Moortel for encouraging discussions on the topic of this manuscript.

\addcontentsline{toc}{section}{References}
\printbibliography


\end{document}